\documentclass[12pt,reqno,a4paper]{amsart}
\usepackage{verbatim,rcs,cite}
\usepackage{amssymb,amsfonts,latexsym,mathrsfs}
\usepackage{amsmath}
\usepackage{amsthm}
\usepackage{mathrsfs}
\usepackage[all]{xy}
\usepackage{pstricks}
\usepackage{hyperref}
\usepackage{tikz}

\newcommand{\tmop}[1]{\ensuremath{\operatorname{#1}}}
\newcommand{\assign}{:=}
\newcommand{\nin}{\not\in}

\newtheorem{theorem}{Theorem}[section]
\newtheorem{lemma}[theorem]{Lemma}
\newtheorem{proposition}[theorem]{Proposition}
\newtheorem{corollary}[theorem]{Corollary}

\theoremstyle{remark}
\newtheorem{example}[theorem]{Example}
\newtheorem{remark}[theorem]{Remark}
\theoremstyle{definition}
\newtheorem{definition}[theorem]{Definition}

\title{Topology of eigenspace posets for imprimitive reflection groups}

\date{}

\author{Justin Koonin}

\dedicatory{\upshape
School of Mathematics and Statistics \\
University of Sydney, NSW 2006, Australia\\[.5em]
{\it E-mail address:} \ \texttt{justin.koonin@sydney.edu.au} }

\begin{document}




\thanks{\noindent{The author is supported by ARC Grant \#DP110103451 at the University of Sydney.}}

\begin{abstract}

This paper studies the poset of eigenspaces of elements of an imprimitive unitary reflection group, for a fixed eigenvalue, ordered by the reverse of inclusion.  The study of this poset is suggested by the eigenspace theory of Springer and Lehrer.  The posets are shown to be isomorphic to certain subposets of Dowling lattices (the "$d$-divisible, $k$-evenly coloured Dowling lattices").  This enables us to prove that these posets are Cohen-Macaulay, and to determine the dimension of their top homology.

\end{abstract}

\maketitle

\section{Introduction}

Let $V$ be a complex vector space of finite dimension, and $G \subseteq GL(V)$ a unitary reflection group in $V$.  Let $\mathcal{A}(G)$ be the set of
reflecting hyperplanes of all reflections in $G$, and
$\mathcal{L}(\mathcal{A}(G))$ the poset of of intersections of the hyperplanes in $\mathcal{A}(G)$. This poset is closely connected to the topology of the hyperplane complement of $\mathcal{A}(G)$ (see \cite{Arnold1969}, \cite{Brieskorn1973}, \cite{OrSo1980},
\cite{OrSo1980-2}, \cite{Lehrer1995}, \cite{BlLe2001}).

It is well-known that the poset $\mathcal{L}(\mathcal{A}(G))$ is a geometric
lattice. Hence it is Cohen-Macaulay, and its reduced homology
vanishes except in top dimension. The poset $\mathcal{L}(\mathcal{A}(G))$ is
also known to coincide with the poset of fixed point subspaces (or
1-eigenspaces) of elements of $G$ (see \cite[Theorem 2.11]{Koonin2012-5}). Springer and Lehrer (\cite{Springer1974}, \cite{LeSp1999}, \cite{LeSp1999-2}) developed a general theory of eigenspaces for unitary reflection groups.
This paper is the second in a series (commencing with \cite{Koonin2012-5}) whose purpose is to study topological properties of
generalisations of $\mathcal{L}(\mathcal{A}(G))$ for arbitrary eigenvalues. This paper concentrates on the case where $G$ is the imprimitive unitary reflection group $G (r, p, n)$, but the definitions in the following paragraph remain valid for any unitary reflection group (indeed, for any finite subset of $\tmop{End} (V)$.

Namely, let $\zeta$ be a complex root of unity, and $g$ be an element of $G$. Define $V (g, \zeta)
\subseteq V$ to be the $\zeta$-eigenspace of $g.$ That is, $V (g, \zeta) := \{v \in V \mid g
v = \zeta v\}.$  Let $\mathcal{S}_{\zeta}^V (G)$ be
the set $\{V (g, \zeta) \mid g \in G\}$, partially ordered by the reverse of
inclusion. More generally, if $\gamma \in N_{GL(V)} (G)$ (the
normaliser of $G$ in $GL(V)$) and $\gamma G$ is a reflection
coset, we may define $\mathcal{S}_{\zeta}^V (\gamma G)$ to be the set $\{V (x, \zeta)
\mid x \in \gamma G\}$, partially ordered by the reverse of inclusion.  This is a linear analogue of the poset of
$p$-subgroups of a group $G$ first studied by Quillen \cite{Quillen1978}.  The study of this poset was first suggested by Lehrer (see \cite[Appendix C, p.270]{LeTa2009}).

The paper commences with a discussion of wreath products
and the infinite family of imprimitive unitary reflection groups $G (r, p,
n)$. Following this is a combinatorial section on Dowling lattices. In
particular we introduce the ``$d$-divisible, $k$-evenly coloured'' Dowling
lattices, which to our knowledge have not been studied before.

An important result is Theorem \ref{theorem:dowling}, which describes an
isomorphism between the poset $\mathcal{S}_{\zeta}^V (\gamma G)$, where $G$ is
an imprimitive unitary reflection group, and these sublattices of Dowling
lattices.

From this isomorphism it follows, in the imprimitive case, that
the poset $\mathcal{S}_{\zeta}^V (\gamma G)$ is Cohen-Macaulay (Corollary \ref{corollary:CM}).  This is perhaps the key result of this paper, which is used to prove a similar statement for arbitrary unitary reflection groups in \cite{Koonin2012-5}.

Given the poset $\mathcal{S}_{\zeta}^V (\gamma G)$, remove the unique maximal element (which
will be shown to always exist), and the unique minimal element (if it exists), and
denote the resulting poset $\widetilde{\mathcal{S}}_{\zeta}^V (\gamma G)$. This construction is necessary since any poset with a unique minimal or maximal element is contractible, so its homology is uninteresting.  Since
$\widetilde{\mathcal{S}}_{\zeta}^V (\gamma G)$ is also Cohen-Macaulay when $G =
G (r, p, n)$, its reduced homology vanishes except in top dimension. When
$G$ is an imprimitive reflection group, the posets $\mathcal{S}_{\zeta}^V
(\gamma G)$ are in fact examples of exponential Dowling structures, a class of
posets defined by Ehrenborg and Readdy \cite{EhRe2009} as generalisations of
Stanley's exponential structures \cite{Stanley1979}. This enables us to find
generating functions for the M\"obius function of these posets, and hence for the dimension of the top homology.  This has applications to the representation theory of unitary reflection groups.

The notation and preliminary material for this paper can be found in \cite[\S2]{Koonin2012-5}.  Rather than repeat material here, we refer the reader to that paper. 

\section{Shellability}

A useful tool for establishing Cohen-Macaulayness of posets or simplicial
complexes is shellability. The original theory applied only to pure
complexes; a non-pure generalisation was given by Bj\"orner and Wachs
\cite{BjWa1996}. The posets in this thesis are all pure, but it is no more
difficult to give the general definition. A good reference for this material
is \cite{Wachs2004}, and this treatment is taken from there.

\begin{definition}{\cite[\S3.1]{Wachs2004}} For each simplex $F$ of a simplicial
  complex $\Delta$, let $\langle F \rangle$ denote the subcomplex generated by
  $F$ (i.e. $\langle F \rangle =\{G \in \Delta \mid G \subseteq F\}) .$  Then
  $\Delta$ is said to be {\em shellable} if its facets (maximal faces) can be
  arranged in linear order $F_1, F_2, \ldots, F_t$ in such a way that the
  subcomplex $\left( \bigcup_{i = 1}^{k - 1} \langle F_i \rangle \right) \cap
  \langle F_k \rangle$ is pure and $(\dim F_k - 1)$-dimensional for all $k =
  2, \ldots, t$. Such an ordering of facets is called a shelling. 
\end{definition}

Recall that if $\Delta$ is a pure complex of dimension $n$, all facets have
dimension $n$, so that $\left( \bigcup_{i = 1}^{k - 1} \langle F_i \rangle
\right) \cap \langle F_k \rangle$ must have dimension $(n - 1)$ for all $k =
2, 3, \ldots, t$.

This definition carries over to posets if we replace simplices
by chains. Unlike the property of being Cohen-Macaulay, shellability is not
a topological property. As an example, there exist nonshellable
triangulations of the 3-ball and the 3-sphere (see \cite[\S3.1]{Wachs2004} for
more details).

In order to establish shellability of a simplicial complex, Bj\"orner
and Wachs developed the notion of lexicographic shellability for bounded
posets, first for pure posets (\cite{BjWa1982}, \cite{BjWa1983}) and later in the
general situation (\cite{BjWa1996}, \cite{BjWa1997}). There are two basic versions of
lexicographic shellability - edge labelling (EL-) shellability and chain
labelling (CL-) shellability. It is the second of these which concerns us
here.

\begin{definition}
  Let $P$ be a bounded poset, and $\Lambda$ any poset. Let
  $\mathcal{M}\mathcal{E}(P)$ be the set of pairs $(c, x \prec y)$, where $c$ is
  a maximal chain in $P$ and $x \prec y$ is an edge along that chain. A
  {\em chain-edge labelling} of $P$ is a map $\lambda : \mathcal{M}\mathcal{E}(P)
  \rightarrow \Lambda$ which satisfies the following condition:

Suppose two maximal chains coincide along their bottom $d$ edges, for some $d$. Then their labels coincide along those edges.
\end{definition}

The following example of a chain-edge labelling is found in \cite[Figure 2.2]{BjWa1983}
(and \cite[Figure 3.3.1]{Wachs2004}). In this case, $\Lambda =\mathbb{Z}$. Note that the top left edge has two labels -- it has the label 3 if paired with
the leftmost maximal chain, and 1 if paired with the other maximal chain which
shares this edge.

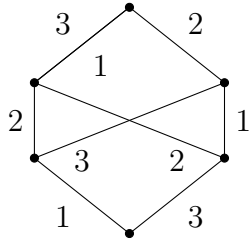
\begin{center}
\begin{tikzpicture}
\filldraw(-0.25,1)circle(1.5pt);
\filldraw(-0.25,2)circle(1.5pt);
\filldraw(1,0)circle(1.5pt);
\filldraw(2.25,1)circle(1.5pt);
\filldraw(2.25,2)circle(1.5pt);
\filldraw(1,3)circle(1.5pt);
\draw(-0.25,1)--node[below left]{1}(1,0);
\draw(1,0)--node[below right]{3}(2.25,1);
\draw(2.25,1)--node[right]{1}(2.25,2);
\draw(2.25,2)--node[above right]{2}(1,3);
\draw(1,3)--node[above left]{3}(-0.25,2);
\draw(1,3)--node[below right]{1}(-0.25,2);
\draw(-0.25,2)--node[left]{2}(-0.25,1);
\draw(-0.25,1)--node[below]{3}(1,1.5);
\draw(1,1.5)--(2.25,2);
\draw(-0.25,2)--(1,1.5);
\draw(1,1.5)--node[below]{2}(2.25,1);
\end{tikzpicture}
\end{center}

Given a chain-edge labelling and a maximal chain $c = ( \hat{0}
\prec x_1 \prec \ldots \prec x_t \prec \hat{1}$) it is possible to form a
word
\[\lambda (c) = \lambda (c, \hat{0} \prec x_1) \lambda
(c, x_1 \prec x_2) \ldots \lambda (c, x_t \prec \hat{1}).
\]

We say that $c$ is {\em increasing} if the
associated word $\lambda (c)$ is {\em strictly} increasing. That is, $c$ is increasing if
\[
\lambda (c, \hat{0} \prec x_1) < \lambda (c, x_1 \prec
x_2) < \ldots < \lambda (c, x_t \prec \hat{1}) \quad\mbox{in $\Lambda$.}
\]

We say that $c$ is {\em decreasing} if the
associated word $\lambda (c)$ is {\em weakly} decreasing. That is, $c$ is decreasing if
\[
\lambda (c, \hat{0} \prec x_1) \geqslant \lambda (c, x_1
\prec x_2) \geqslant \ldots \geqslant \lambda (c, x_t \prec \hat{1}) \quad\mbox{ in
$\Lambda$}.
\]
We order the maximal chains lexicographically by using the
lexicographic order on the corresponding words.

To define chain-lexicographic labelling, it is necessary to restrict
chain-edge labellings $\lambda : \mathcal{M}\mathcal{E}(P) \rightarrow
\Lambda$ to closed intervals $[x, y]$. However as seen above each edge may
have more than one label, depending upon which maximal chain containing that
edge we choose. Given the condition on the chain-edge labelling, each
maximal chain $r$ of $[\hat{0}, x]$ determines a unique restriction of
$\lambda$ to $\mathcal{M}\mathcal{E}([x, y]) .$

\begin{definition}
  Let $P$ be a bounded poset. Define a closed rooted interval $[x, y]_r$ to be
  a closed interval $[x, y]$, together with a maximal chain $r$ of $[ \hat{0},
  x]$.
\end{definition}

\begin{definition}\label{definition:CLlabelling}
  Let $P$ be a bounded poset. A {\em chain-lexicographic labelling (CL-labelling)} of $P$ is a chain-edge labelling
  such that in each closed rooted interval $[x, y]_r$ of $P$, there is a unique
  strictly increasing maximal chain, which lexicographically precedes all
  other maximal chains of $[x, y]_r$. If a poset admits a CL-labelling it is
  said to be {\em CL-shellable}.
\end{definition}

The chain-edge labelling in the figure above is a CL-labelling. The unique
increasing maximal chain for the whole poset is the leftmost maximal chain.

The following implications between the concepts we have discussed are noted in \cite{Wachs2004}:

\begin{proposition}{\cite[Lecture 4.1]{Wachs2004}}\label{proposition:CMimplicationswachs} 
The following implications hold for a poset $P$:
\begin{align*}
\mbox{(pure) CL-shellable} &\Rightarrow \mbox{(pure) shellable}\\
&\Rightarrow \mbox{homotopy Cohen-Macaulay}\\
&\Rightarrow \mbox{Cohen-Macaulay over $\mathbb{Z}$}\\
&\Rightarrow \mbox{Cohen-Macaulay over $\mathbb{F}$ (for any field
$\mathbb{F}$)}
\end{align*}
\end{proposition}

It is known that all these implications are strict. For a
further discussion see \cite[Appendix]{Bjorner1980}, \cite[\S9]{BjWa1983}. We
note in passing that a similar chain of implications holds in the non-pure
case, when each of the above properties is replaced by its non-pure analogue.

We now present a technical tool for establishing CL-shellability, called
{\em recursive atom ordering}. It is equivalent to
CL-shellability, but does not involve edge labellings.

\begin{definition}\label{def:recursive}
  A bounded poset $P$ is said to admit a
  recursive atom ordering if its length $l (P)$ is 1, or if $l (P) > 1$ and
  there is an ordering $a_1, a_2, \ldots, a_t$ of the atoms of $P$ that
  satisfies
  \begin{itemize}
\item[(i)] For all $j = 1, 2, \ldots, t$ the interval $[a_j, \hat{1}]$
  admits a recursive atom ordering in which the atoms of $[a_j, \hat{1}]$ that
  belong to $[a_{i,} \hat{1}]$, for some $i < j$, come first.
\item[(ii)] For all $i < j$, if $a_i, a_j < y$ then there is a $k <
  j$ and an atom $z$ of $[a_j, \hat{1}]$ such that $a_k < z \leqslant y.$
\end{itemize}
A {\em recursive coatom ordering} is a recursive atom ordering of the dual
  poset $P^*$.
\end{definition}

A key theorem relating recursive atom orderings to shellability, which will be
used later is the following:

\begin{theorem}{\cite{BjWa1983},\cite{Wachs2004}}\label{theorem:recursive} 
A bounded poset $P$ is CL-shellable if and only if $P$ admits a recursive atom ordering.
\end{theorem}

In particular, it is known (see \cite[Theorem 4.2.7]{Wachs2004}) that every
geometric lattice admits a recursive atom ordering, and is therefore
CL-shellable.

\section{The Groups $G (r, p, n)$}\label{section:G(r,p,n)}

\subsection{Wreath Products}\label{subsection:wreath}

The following treatment is taken from \cite[{\S}3.2]{LeTa2009}.

Given a group $G$ acting on the set $\Omega \assign \{1, 2, \ldots, n\}$ and
any group $H$, let $B$ denote the direct product of $n$ copies of $H$, and
define an action of $G$ on $B$ as follows. If $g \in G$ and $h = (h_1, h_2,
\ldots, h_n) \in B$, define
\[
g.h \assign (h_{g^{- 1} (1)}, h_{g^{- 1} (2)},
\ldots, h_{g^{- 1} (n)}).
\]
The {\em wreath product} $H$ by $G$, denoted $H \wr G$, is the semidirect
product of $B$ by $G$. Write
elements of $H \wr G$ as $(h ; g)$, with multiplication given by
\[
(h ; g) (h' ; g') = (hg.h' ; gg'),\quad\mbox{for all } g\in G, h\in H.
\]
If $H$ and $G$ are finite groups, the order of $H \wr G$ is $\left| H
\right|^n \left| G \right| .$

Assume now that $H$ is the cyclic group $\mu_m$ of complex $m$-th roots of
unity, for some $m \in \mathbb{N}$. Let $V$ be a complex vector space of
dimension $n$, equipped with positive definite hermitian form $(-, -)$, and
orthonormal basis $e_1, e_2, \ldots, e_n$. The elements of $B$ can be
represented as diagonal transformations of $V$. That is, if $h = (h_1,
h_2, \ldots, h_n) \in B$, set
\[
he_i : = h_i e_i\quad\mbox{for}\quad 
\leqslant i \leqslant n.
\]
Clearly $B$ is a subgroup of $U (V) .$

$G$ also acts on $V$, by permutation of the basis vectors:
\[
ge_i : = e_{g (i)} \quad\mbox{for}\quad 1
\leqslant i \leqslant n.
\]
The actions of $B = (\mu_m^{})^n$ and $G$ combine to
give an action of $\mu_m \wr G$ on $V$:
\[
(h ; g) e_i \assign h_{g (i)} e_{g
(i)}\quad\mbox{for} \quad 1 \leqslant i \leqslant n.
\]
The matrices representing these linear transformations (with respect to the
basis $\{e_1, \ldots, e_n \})$ are monomial (one nonzero entry in each row and
each column) and this representation of $\mu_r \wr G$ is known as the
{\em standard monomial representation}. Note that $\mu_r \wr G$ is in
fact represented as a subgroup of $U (V)$.

With $B = (\mu_r)^n$ and $p$ a divisor of $r$, define the group $A
(r, p, n)$ as follows:
\[
A (r, p, n) \assign \{(\theta_1, \theta_2, \ldots,
\theta_n) \in B \mid (\theta_1 \theta_2 \ldots \theta_n)^{\frac{r}{p}} = 1\}.
\]
The group $G (r, p, n)$ is defined to be the semidirect product of $A (r, p,
n)$ by the symmetric group $\mathrm{Sym}(n)$. The group $G (r, p, n)$ is a
normal subgroup of index $p$ in $\mu_r \wr \tmop{Sym} (n)$, so it can be
represented as a group of monomial linear transformations in the standard
monomial representation. The order of $G (r, p, n)$ is $\frac{r^n n!}{p}$.

 The group $G (r, p, n)$ is defined to be the group
of all $n \times n$ monomial matrices whose non-zero entries are complex $r$-th
roots of unity such that the product of the non-zero entries is an
$\frac{r}{p}$-th root of unity.

The following results about the groups $G (r, p, n)$ are standard:

\begin{proposition}{\em (\cite[Theorem 2.4, Remark 2.5]{Cohen1976}, \cite[Proposition 2.10, Theorems 2.14 and 2.15]{LeTa2009}), \cite[Proposition 2.1.6]{Blair1997}.} Let $V$ be a
  complex vector space of dimension $n$, equipped with positive definite
  hermitian form, and denote the unitary group of the form by $U(V)$. Let $r,p$
  and $n$ be positive integers such that $p$ divides $r$. Then:
  \begin{itemize}
  \item[(i)] $G (r, p, n)$ is a unitary reflection group.
  \item[(ii)] Suppose $n \geq 2$.  If $(r_1, p_1, n)$ and $(r_2, p_2, n)$ are triples of positive
  integers with $p_1 \mid r_1$ and $p_2 \mid r_2$ then $G (r_1, p_1, n)$ and
  $G (r_2, p_2, n)$ are conjugate in U(V) if and only if $(r_1, p_1, n) =
  (r_2, p_2, n)$ or $\{(r_1, p_1, n), \text{$(r_2, p_2, n)\}$} $= $\{(2,
  1, 2), (4, 4, 2)\}$.
\item[(iii)] $G (r, p, n)$ is an irreducible unitary reflection group except when
  $(r, p, n) = (1, 1, n)$ and $n>1$  or $(r, p, n) = (2, 2, 2)$.  The group $G(1, 1, n)$ is the symmetric group on $n$ letters, which splits as a sum of irreducible representations of dimensions $1$ and $n-1$.  The group $G(2, 2, 2                                                                                                                                                                                                                                                                                                                                            )$ is the Klein four group.
  \end{itemize}
\end{proposition}

\subsection{Dowling Lattices}\label{subsection:Dowling}

When $G$ is an imprimitive reflection group, the posets $\mathcal{S}_{\zeta}^V
(\gamma G)$ turn out to be isomorphic to subposets of the well-known
{\em Dowling lattices}. This section provides an introduction to
Dowling lattices. This treatment follows \cite{GoWa2000} closely, with an
adaptation from \cite{Henderson2006}.

Let $\pi = B_1 \mid \ldots \mid B_k$ denote a partition of a set $S$. This
means that $S$ is the disjoint union of the $B_i$, which are called
{\em blocks} of $\pi$. Write $\left| \pi \right|$ for the number of blocks
in $\pi$. 

Let $G$ be a finite group. A {\em $G$-prepartition} is a triple $(I, \pi,
\gamma)$ where $I \subseteq [n] =\{1, \ldots, n\},$ $\pi = B_1 \mid \ldots
\mid B_k$ is a partition of $[n] \setminus I$ and $\gamma$ is a map $B_1 \cup
B_2 \cup \ldots \cup B_k \rightarrow G.$ The subset $I$ is called the
{\em zero block} of the prepartition. For $i = 1, 2, \ldots, k$ the restriction of $\gamma$ to $B_i$ is denoted $\gamma_i$. The pair $(B_i, \gamma_i),$ $i = 1, 2,
\ldots k$ is called a {\em nonzero block}, and can be considered as a subset $\{(a, \gamma_i
(a)) \mid a \in B_i \}$ of $[n] \times G$. Thus, $\gamma$ may be thought of
as a labelling of the elements of the nonzero blocks by elements of $G$. Note that the zero block is allowed to be empty, but that nonzero blocks are,
by definition, nonempty.

There is a right action of $G$ of $[n] \times G$ given by $(i, h) \cdot g =
(i, h g)$. This action extends to an action on the set of all subsets of
$[n] \times G$ by $S \cdot g =\{x \cdot g \mid x \in S\}.$ There is an equivalence relation on the power set of $[n] \times G$ given by setting two subsets $S_1$ and $S_2$ of $[n] \times G$ to be equivalent if they
lie in the same $G$-orbit under this action. 

There is also an equivalence relation on the set of $G$-prepartitions.
Two $G$-prepartitions are equivalent if their zero blocks
are equal and there is a one-to-one correspondence between their nonzero
blocks such that corresponding blocks are equivalent. 

Define a {\em $G$-partition} to be an equivalence class of $G$-prepartitions.  If $(I, \pi, \gamma)$ is a member of this equivalence class, the corresponding $G$-partition is denoted $\overline{(I,
\pi, \gamma)}$. The equivalence class of a nonzero block $(B_i, \gamma_i)$
of a $G$-prepartition is called a {\em nonzero $G$-block} of the
$G$-partition, and is denoted $\overline{(B_i, \gamma_i)}$. The
zero block $I$ of the $G$-prepartition is called the
 {\em zero $G$-block} of the $G$-partition. For ease of notation we write $\bar{I} = I.$

The elements of the Dowling lattice $Q_n (G)$ are the set of $G$-partitions.  The partial order is determined by the covering relation described as follows. First
define a {\em merge} of two $G$-blocks. The merge of the zero
$G$-block $I$ with a nonzero $G$-block
$\overline{(B_i, \gamma_i)}$ is the zero block $I \cup B_i \subseteq [n]$.
There are $\left|G\right|$ ways to merge two nonzero blocks $\overline{(B_i,
\gamma_i)}$ and $\overline{(B_j, \gamma_j)},$ each of which has the form
$\overline{(B_i, \gamma_i) \cup ((B_j, \gamma_j) \cdot g)}$ for some $g \in
G$ where the union is of two subsets of $[n] \times G$. Note that
once equivalence class representatives are fixed
for the $G$-blocks being merged, the merge is defined uniquely by $g$. The covering relation can now be simply defined: in $Q_n (G)$, the $G$-partition
$\overline{(I, \pi, \gamma)}$ is covered by $\overline{(I', \pi', \gamma')}$
if $\overline{(I', \pi', \gamma')}$ can be obtained from $\overline{(I, \pi,
\gamma)}$ by merging two $G$-blocks of $\overline{(I, \pi, \gamma)}$. The
minimum element $\hat{0}$ of $Q_n (G)$ is the $G$-partition in which
$I = \emptyset$, and $\pi$ partitions the set $\{1, \ldots, n\}$ into
singleton subsets (so there is no choice for $\gamma$, up to equivalence). The
maximum element $\hat{1}$ of $Q_n (G)$ is the $G$-partition in which
$I =\{1, \ldots, n\}$, $\pi$ is empty, and $\gamma$ is trivial. 

There is a left action of $G \wr \tmop{Sym} (n)$ on $Q_n (G)$ defined as
follows. Suppose $(i, a) \in [n] \times G$, $g = (g_1, \ldots, g_n) \in G^n$ and $\sigma \in \tmop{Sym} (n)$. Thus $(g ; \sigma) \in G \wr
\tmop{Sym} (n)$. Define
\[
(g ; \sigma) \cdot (i, a) = (\sigma (i), g_{\sigma
(i)} a).
\]

This induces an action of $G \wr \tmop{Sym} (n)$ on subsets of $[n] \times G$
and therefore on $G$-prepartitions, by letting an element $(g ; \sigma) \in G
\wr \tmop{Sym} (n)$ act on each nonzero block viewed as a subset of $[n]
\times G$, and on the zero block $I =\{i_1, i_2, \ldots, i_k \}$ by $(g ;
\sigma) \cdot \{i_1, i_2, \ldots, i_k \}=\{\sigma (i_1), \sigma (i_2), \ldots,
\sigma (i_k)\}.$ Notice that the left action of $G \wr \tmop{Sym} (n)$ on
nonzero blocks commutes with the right action of $G$. Therefore the left
action of $G \wr \tmop{Sym} (n)$ on $G$-prepartitions respects the equivalence
relation on $G$-prepartitions, and hence induces an action on $G$-partitions. This action preserves the partial order on $Q_n (G)$.

The Dowling lattice  $Q_n (G)$ is ranked, with rank function given by
\[
\tmop{rk}(\overline{(I, \pi, \gamma)}) = n -
\left| \pi \right|.
\]

Thus the rank of the whole poset $Q_n (G)$ is $n$.

It is known (see \cite[Theorem 3]{Dowling1973}) that $Q_n (G)$ is a geometric
lattice, and therefore Cohen-Macaulay.

Now specialise to the case where $G$ is the cyclic group $\mu_r$, identified
with $\mathbb{Z}/ r\mathbb{Z}$. Let $Q_n (r)$ denote the Dowling lattice
$Q_n (\mu_r) .$ When $G$ is the trivial group $\mu_1$ the Dowling
lattice $Q_n (1)$ is the ordinary partition lattice $\Pi_{n + 1}$ on $n + 1$
elements, and $\mu_1 \wr \tmop{Sym} (n)$ is just the symmetric group
$\tmop{Sym} (n)$. The action of $\mu_1 \wr \tmop{Sym} (n)$ on $Q_n (1)$ is
isomorphic to the restriction of $\tmop{Sym} (n + 1)$ on $\Pi_{n + 1}$
to $\tmop{Sym} (n)$. When $G = \mu_2$, the Dowling lattice $Q_n (2)$ is the
signed partition lattice $\Pi_n^B$. The wreath product $\mu_2 \wr \tmop{Sym}
(n)$ is the hyperoctahedral group. The lattice $Q_n (1)$ is the intersection
lattice for the Coxeter group of type $A_n$, while $Q_n (2)$ is the
intersection lattice for the Coxeter group of type $B_n$.

It is well known (see Proposition \ref{prop:hyperplanearrangement}) that the
Dowling lattice $Q_n (r)$ is the intersection lattice for the imprimitive
reflection group $G (r, 1, n)$.

Several subposets of $Q_n (r)$ occur in the work which follows:

\begin{definition}\label{definition:Q_n(r,J)}
  Let $J \subseteq \{0, 1, \ldots, n\}$.  The poset $Q_n (r, J)$ is the subposet of $Q_n (r)$ whose  elements $\overline{(I, \pi, \gamma)}$ satisfy $\left| I \right| \nin J \subseteq
  \{0, 1, \ldots, n\}$, together with the unique minimal element of $Q_n (r)$.
\end{definition}

Note that this definition applies more generally to the Dowling lattice
$Q_n (G)$.

\begin{definition}\label{definition:(Q_n(r,d)}
  The poset $Q_n (r, d)$ is the subposet of $Q_n (r)$ whose
  nonzero blocks have length divisible by $d$, together with the unique
  minimal element of $Q_n (r)$ (the `$d$-divisible Dowling lattice').
  As a matter of convention we allow the case $d = 1$, in which case $Q_n
  (r, d) = Q_n (r)$.
\end{definition}

\begin{definition}\label{definition:Q_n(r,d,k)}
  The poset $Q_n (r, d, k)$ is the subposet of $Q_n (r, d)$
  consisting of elements $\overline{(I, \pi, \gamma)}$ such that for each
  nonzero block $\overline{(B_i, \gamma_i)}$, $\left| \{a \in B_i \mid \gamma
  (a) \equiv j \;(\tmop{mod} \;k)\} \right|$ is independent of $j$, together with the
  minimum element of $Q_n(r, d)$ (the `$k$-evenly
  coloured, $d$-divisible Dowling lattice'). Note that we require
  $k \mid r$ and will also assume that $k \mid d$. 
\end{definition}

\begin{definition}
  Let $J \subseteq \{0, 1, \ldots, n\}$.  The poset $Q_n (r, d, k, J)$ is the subposet\label{definition:Q_n(r,d,k,J)} of $Q_n
  (r, d, k)$ whose elements $\overline{(I, \pi, \gamma)}$ satisfy $\left| I \right|
  \nin J \subseteq \{0, 1, \ldots, n\}$, together with the unique minimal element of $Q_n (r, d, k)$.
\end{definition}

We would like to show that these posets are Cohen-Macaulay over $\mathbb{Z}$.  It is well-known (see \cite{BjWa1983},\cite[Theorem 4.2.2]{Wachs2004}) that it suffices to show that the posets in question have a recursive atom ordering:  

\begin{definition}\label{def:recursive}{\cite[Definition 4.2.1]{Wachs2004}}
  A bounded poset $P$ is said to admit a
  recursive atom ordering if its length $l (P)$ is 1, or if $l (P) > 1$ and
  there is an ordering $a_1, a_2, \ldots, a_t$ of the atoms of $P$ that
  satisfies
  \begin{itemize}
\item[(i)] For all $j = 1, 2, \ldots, t$ the interval $[a_j, \hat{1}]$
  admits a recursive atom ordering in which the atoms of $[a_j, \hat{1}]$ that
  belong to $[a_{i,} \hat{1}]$, for some $i < j$, come first.
\item[(ii)] For all $i < j$, if $a_i, a_j < y$ then there is a $k <
  j$ and an atom $z$ of $[a_j, \hat{1}]$ such that $a_k < z \leqslant y.$
\end{itemize}
\end{definition}

The following proposition regarding the topological properties of some of
these posets appears in \cite{Henderson2006}. Note that the result looks slightly
different because we have defined Dowling lattices differently.

\begin{proposition}{\cite[Proposition 1.6]{Henderson2006}}\label{proposition:Qnrd}
   The poset $Q_n (r, d)$ is a pure lattice with a recursive atom
  ordering. Its rank function is
\[
\tmop{rk}(\overline{(I, \pi, \gamma)}) = \left\{\begin{array}{ll} 0,&
  \mbox{if $\overline{(I, \pi, \gamma)} = \hat{0}$}\\
\lfloor \frac{n}{d} \rfloor + 1 - \left| \pi \right|,&\mbox{otherwise.}\end{array}\right.
 \]
 \end{proposition}  
  
  \begin{proof}
    A proof is given for completeness, using the definition of the Dowling
    lattice given here, rather than that in \cite{Henderson2006}.

    The poset $Q_n (r, d) \setminus \{ \hat{0} \}$ is an upper order ideal of
    $Q_n (r)$, so $Q_n (r, d)$ is a lattice. The atoms of $Q_n (r, d)$ are
    those $\overline{(I, \pi, \gamma)}$ such that $\left| B_i \right| = d$ for each
    block $B_i$ of $\pi$, and $\left| \pi \right| = \lfloor \frac{n}{d} \rfloor .$
    These all have rank $(n - \lfloor \frac{n}{d} \rfloor)$ as elements of
    $Q_n (r)$, so $Q_n (r, d)$ is pure with the required rank function. To
    find a recursive atom ordering, note that for any element $ \overline{(I,
    \pi, \gamma)} \in Q_n (r, d) \setminus \{ \hat{0} \}$, the principal
    upper order ideal $[ \overline{(I, \pi, \gamma)}, \widehat{1}]$ is
    isomorphic to $Q_{\left| \pi \right|} (r)$, which is a geometric lattice. Thus $[ \overline{(I, \pi, \gamma)}, \widehat{1}]$ is CL-shellable, and
    hence admits a recursive atom ordering. Hence (see Definition
    \ref{def:recursive}) it suffices to check that the atoms of $Q_n (r, d)$
    can be ordered $a_1, a_2, \ldots, a_t$ in such a way that for all $i < j$
    and $y \in Q_n (r, d) \backslash \{ \hat{0} \}$, if $a_i, a_j < y$ then there is a $k < j$ and an atom $z$ of $[a_j, \hat{1}]$ such that $a_k < z \leqslant y.$

    Such an ordering can be described as follows (see \cite[Exercise 4.3.6(a)]{Wachs2004} and \cite[Proposition 1.6]{Henderson2004}). Let $a = \overline{(I,
    \pi, \gamma)}$ be an atom of $Q_n (r, d)$. This atom is characterised by
    its zero block $I$ of size $n - d \lfloor \frac{n}{d} \rfloor$ and
    $\lfloor \frac{n}{d} \rfloor$ nonzero blocks $(B_{1,} \gamma_1), \ldots,
    (B_{\lfloor \frac{n}{d} \rfloor}, \gamma_{\lfloor \frac{n}{d} \rfloor})$,
    where each subset $B_i$ of $[n] \setminus I$ has $d$ elements. List the
    elements of $I$ and of each block in increasing order, ignoring the labels
    determined by $\gamma$. Call these lists $\overline{I}, \overline{B_1},
    \ldots, \overline{B_{\lfloor \frac{n}{d} \rfloor}} .$ Assume that the
    ordering of $\overline{B_1}, \ldots, \overline{B_{\lfloor \frac{n}{d}
    \rfloor}} $ themselves is determined by the\, smallest element in each of
    these lists. Form the word $\alpha = \overline{I} \mid \overline{B_1}
    \mid \cdots \mid \overline{B_{\lfloor \frac{n}{d} \rfloor}}$. This
    process associates a word of length $n$ to each atom $a$.

    As an example, suppose $n = 8$, $r = 2$, $d = 3$. Since there are only two
    possible labelings for each element, we can use the presence or absence
    of a bar ($\overline{\;\;}$) above an element to indicate this labelling. Use a vertical line $(\;\mid\;)$ to separate nonzero blocks, and a double line
    $(\;\parallel\;)$ to separate the zero block from the nonzero blocks. Thus,
    a typical atom $a$ might be $a = 1 \bar{4} 6 \mid \bar{2} 78
    \parallel 35$. According to the above definition, the word $\alpha$
    associated with $a$ is $\alpha = 35146278.$

    Now order the atoms of $Q_n (r, d)$ by lexicographic order of these
    words. It is possible that different atoms may produce the same word. For example, the atom $a' = 146 \mid 2 \bar{7} 8 \parallel 35$ also has
    the word $\alpha = 35146278$ associated to it. Within atoms with the
    same word, choose an arbitrary ordering of all $G$-blocks of size $d$ with
    underlying set $B$, for each $d$-element subset $B$ of $[n]$. Apply this
    ordering first to $(B_1, \gamma_1),$ then to $(B_2, \gamma_2)$, and so on.
    Thus the dependence of the ordering on $\gamma$ is arbitrary.
    
Now suppose that $a_j = \overline{(I, \pi, \gamma)}$, $y = \overline{(I',
    \pi', \gamma')}$, and that $a_j < y.$ Write the nonzero blocks of $a_j$
    as $(B_{1,} \gamma_1), \ldots, (B_{\lfloor \frac{n}{d} \rfloor},
    \gamma_{\lfloor \frac{n}{d} \rfloor})$, and those of $y$ as $(B_1',
    \gamma_1'), \ldots, (B_k', \gamma_l')$ for some $l$. Suppose that
    there is no $k$ such that there is a common cover $z$ of both $a_k$ and
    $a_j$, and that $z \leqslant y.$ It will suffice to show that $a_j$ is
    the earliest atom which is less than $y$. Suppose $B'_i$ is any part of
    $\pi'$. Then $B_i' = B_{s_1} \cup B_{s_2} \cup \cdots \cup B_{s_t}$ for
    some $s_1 < s_2 < \cdots < s_t$. Suppose that for some $q$, $a = \max
    (B_{s_q}) > \min (B_{s_{q + 1}}) = b.$ Consider the atom $a_k  =
    \overline{(I'', \pi'', \gamma'')}$ defined by
\begin{align*}
I'' &= I\\
B_{s_q}'' &= B_{s_q} + b - a \quad\mbox{(that
    is, include $b$ in $B_{s_q}''$ but exclude $a$)}\\
B_{s_{q + 1}}'' &= B_{s_{q + 1}} + a- b\\
B_j'' &= B_j\quad\mbox{for $j \neq s_q, s_{q +1}$}\\
\gamma'' &= \gamma
\end{align*}
Then $a_k$ occurs earlier in the ordering than $a_j$, and their join $z$
    is a cover of both $a_k$ and $a_j$ which is $\leqslant y$. This
    contradicts the assumption on $a_j$. As a result, it must be the case
    that $\max (B_{s_i}) < \min (B_{s_{i + 1}})$. Hence the parts of $\pi$
    contained in $B_i'$ are simply those obtained by ordering the elements of
    $B_i'$ and dividing the list into $d$-element subsets. The labelling on
    each of these parts is determined by the labelling on $B_i'$. A similar
    argument shows that $I$ must consist of the smallest $n - d \lfloor
    \frac{n}{d} \rfloor$ elements of $I'$, and that the parts of $\pi$
    contained in $I'$ can be found by listing the remaining elements of $I'$
    in ascending order, and dividing this list into $d$-element subsets. The
    labelling on each of these parts is determined by choosing the smallest
    element in the (fixed, but arbitrary) ordering of all $G$-blocks of size
    $d$, for each $d$-element subset $B$ of $[n]$. From the construction,
    it is clear that this $a_j$ is the smallest atom less than or equal to
    $y$. This completes the proof.\end{proof}
  
  A similar result holds for the poset $Q_n (r, d, k)$:
  
  \begin{proposition}\label{proposition:Qnrdk}
The poset $Q_n (r, d, k)$ is a pure lattice with
    a recursive atom ordering. Its rank function is
\[
\tmop{rk}(\overline{(I, \pi, \gamma)}) = \left\{\begin{array}{ll}0,&\mbox{if $\overline{(I, \pi, \gamma)}= \hat{0}$}\\
\lfloor \frac{n}{d} \rfloor + 1 - \left| \pi \right|,&\mbox{otherwise.}\end{array}\right.
\]
\end{proposition}
    \begin{proof}
      $Q_n (r, d, k) \setminus \{ \hat{0} \}$ is an upper order ideal of $Q_n (r, d_{})$, generated by those atoms of $Q_n
      (r, d)$ which are $k$-evenly coloured. Hence
      $Q_n (r, d, k)$ is a lattice. The atoms of $Q_n (r, d, k)$ are precisely the $k$-evenly coloured atoms of $Q_n (r,
      d)$. These all have the same rank, and so $Q_n (r, d,
      k)$ is pure, with the same rank function as $Q_n (r, d)$.
      
To see that $Q_n (r, d, k)$ has a recursive atom ordering,
      use the same lexicographic ordering described in Proposition
      \ref{proposition:Qnrd}. As in the proof of Proposition
      \ref{proposition:Qnrd}, suppose that $a_j = \overline{(I, \pi, \gamma)}$, $y = \overline{(I', \pi', \gamma')}$, and that $a_j < y.$ Write the
      nonzero blocks of $a_j$ as $(B_1, \gamma_1), \ldots, (B_{\lfloor
      \frac{n}{d} \rfloor}, \gamma_{\lfloor \frac{n}{d} \rfloor})$, and those
      of $y$ as $(B_1', \gamma_1'), \ldots, (B_k', \gamma_l')$ for some $l$.

      Suppose that there is no $k$ such that there is a common cover $z$
      of both $a_k$ and $a_j$, and that $z \leqslant y.$ It suffices to show
      that $a_j$ is the earliest atom which is less than $y$. Suppose $B_i'$
      is a part of $\pi'$, and that $B_i' = B_{s_1} \cup \cdots \cup B_{s_t}$
      for some $s_! < s_2 < \cdots < s_{t.}$ A similar argument to that used
      in Proposition \ref{proposition:Qnrd} shows that the parts $B_{s_j}$ of
      $\pi$ contained in $B_i'$ are obtained as follows. Suppose $B_i' = C_1
      \cup \cdots \cup C_k$, where for $1 \leqslant j \leqslant k$ the
      elements in $C_j$ all have the same colouring. Order each of the sets
      $C_i$ by lexicographic order. Then $B_{s_1}$ contains the smallest
      $\frac{d}{k}$ elements of each of the sets $C_i$, $B_{s_2}$ contains the
      next smallest $\frac{d}{k}$ elements of each of the $C_i$, and so on. \
      The labelling on each of these parts is determined by the labelling on
      $B_i'$. Similarly, the zero block $I$ consists of the
      smallest $n - d \lfloor \frac{n}{d} \rfloor$ elements of $I'$, and the
      parts of $\pi$ contained in $I'$ can be found by listing the remaining
      elements of $I'$ in ascending order, and dividing this list into
      $d$-element subsets. The labelling on each of these parts is
      determined by choosing the smallest element in the (fixed, but
      arbitrary) ordering of all $G$-blocks of side $d$, for
      each $d$-element subset $B$ of $[n]$. The argument
      is the same as in the proof of Proposition \ref{proposition:Qnrd},
      except that we can only perform the switching algorithm within elements
      of the same colour in each block $B_i'$ of $\pi'$, and we perform this
      algorithm for each colour separately. It is clear that $a_j$ is the
      smallest atom less than or equal to $y$.
    \end{proof}

In addition, in the case $n \equiv 0 \;(\tmop{mod}\; d)$:

\begin{proposition}
      \label{proposition:Qnrdk0} If $n \equiv 0\; (\tmop{mod} \;d)$ then
      $Q_n (r, d, k, \{0\})$ is a pure lattice with recursive atom ordering.
      Its rank function is given by 
      \[
      \tmop{rk}(\overline{(I, \pi, \gamma)}) =\left\{\begin{array}{ll} 0,&\mbox{
     if $\overline{(I, \pi, \gamma)}= \hat{0}$}\\
\lfloor \frac{n}{d} \rfloor - \left| \pi \right|,&\mbox{otherwise.}\end{array}\right.
\]
\end{proposition}
      
\begin{proof}
        The atoms of $Q_n (r, d, k, \{0\})$  consist of $(\lfloor \frac{n}{d}
        \rfloor - 1)$ $k$-evenly coloured nonzero blocks of length $d$, and a
        zero block also of length $d$. $Q_n (r, d, k, \{0\}) \setminus \{
        \hat{0} \}$ is an upper order ideal of $Q_n (r)$ generated by these
        atoms, and therefore $Q_n (r, d, k, \{0\})$ is a lattice. The rank
        function is one less than that of $Q_n (r, d, k)$ as no atom of
        $Q_n (r, d, k)$ is in $Q_n (r, d, k, \{0\})$. The removal of
        elements of $Q_n (r, d, k)$ with zero block of size 0 has does not
        alter the proof of the existence of a recursive atom ordering in
        Proposition \ref{proposition:Qnrdk}, so the same proof indicated above
        applies here.
      \end{proof}
      
      Finally,
      
      \begin{proposition}
 If $n \equiv 1 \;(\tmop{mod}\; d)$ then $Q_n (r, d,
        k, \{1\})$ is a pure lattice with recursive atom ordering. Its rank
        function is given by 
        \[
        \tmop{rk}(\overline{I, \pi, \gamma)}) =\left\{\begin{array}{ll} 0,&\mbox{if $\overline{(I, \pi, \gamma)}= \hat{0}$}\\
        \lfloor \frac{n}{d} \rfloor - \left| \pi \right|,&\mbox{otherwise.}\end{array}\right.
        \]
        \end{proposition}        
\begin{proof}
          The proof is the same as for Proposition \ref{proposition:Qnrdk0}.
        \end{proof}

      Since the posets $Q_n (G)$, $Q_n (r)$, $Q_n (r, d)$, $Q_n (r, d, k)$, $Q_n (r,
      d, k, \{0\})$, $Q_n (r, d, k, \{1\})$ all admit a recursive atom ordering,
      they are all Cohen-Macaulay over $\mathbb{Z}$ (see Proposition \ref{proposition:CMimplicationswachs} and Theorem \ref{theorem:recursive}). Thus the proper part of each of these
      lattices is Cohen-Macaulay (by \cite[Corollary 2.4]{Koonin2012-5}). So for each poset $Q$ in this list, $\widetilde{H}_i ( \bar{Q}, \mathbb{Z})
      = 0$ unless $i = l ( \bar{Q})$.
      
The study of the action of the symmetric group on the partition lattice
      was initiated by Stanley \cite{Stanley1982}, with the famous result that
      $\widetilde{H}_{n - 3} ( \overline{\Pi}_n , \mathbb{C}) \simeq
      \varepsilon_n \otimes \tmop{Ind}_{\mu_n}^{\tmop{Sym} (n)} (\psi)$, where
      $\psi$ is a faithful character of the cyclic group $\mu_n$ generated by
      an $n$-cycle, and $\varepsilon_n$ is the sign representation. Sundaram
      \cite{Sundaram1994} used Whitney homology to provide a conceptual
      proof of Stanley's result (which had relied on earlier computations by
      Hanlon). See also the paper \cite{LeSo1986}.

      For Dowling lattices in general, Hanlon \cite{Hanlon1984} calculated the
      characters of the representations of $G \wr \tmop{Sym} (n)$ on $Q_n
      (G)$. Henderson \cite{Henderson2006} adapted Joyal's theory of tensor
      species to understand the representation of $G \wr \tmop{Sym} (n)$ on
      several subposets of $Q_n (G)$, including (as a special case) $Q_n (r,
      d)$. 

\subsection{Reflection Arrangements for $G (r, p, n)$}

This section describes the hyperplane arrangements for the groups $G (r, p,
n)$. The notation of \cite{OrTe1992} is used throughout. Let $V
=\mathbb{C}^n$, and let $\zeta$ be an $r$-th root of unity.

\begin{definition}
  Define the hyperplanes
\begin{align*}
H_{i, j}^{\zeta} &\assign \{(v_1,
  \ldots, v_n) \in V \mid v_i = \zeta v_j \}\mbox{ and}\\
H_k &\assign \{(v_1, \ldots, v_n)
  \in V \mid v_k = 0\}.
  \end{align*}
\end{definition}

\begin{definition}
  Let
\begin{align*}
\mathcal{A}_n^0 (r) &\assign
  \{H_{i, j}^{\zeta} \mid 1 \leqslant i < j \leqslant n, \zeta \in \mu_r \}\mbox{ and}\\
\mathcal{A}_n (r) &\assign
  \mathcal{A}_n^0 (r) \cup \{H_k \mid 1 \leqslant k \leqslant n\}.
  \end{align*}
\end{definition}

\begin{proposition}\label{proposition:Dowlingisomorphism}
  There is an isomorphism $\mathcal{L}(\mathcal{A}_n (r)) \cong Q_n
  (r)$.
\end{proposition}  
  \begin{proof}
Define a map $\rho : Q_n (r) \rightarrow
    \mathcal{L}(\mathcal{A}_n (r))$ as follows. Suppose $\overline{(I, \pi,
    \gamma)} \in Q_n (r)$, with $\pi = B_1 \mid \ldots \mid B_k$. Let $\rho
    ( \overline{(I, \pi, \gamma)})$ be the subspace defined by the equations
\[
z_i = \left\{\begin{array}{ll}\gamma (i)^{- 1} \gamma
    (j) z_j,&\mbox{if $i, j$ lie in the same block $B_s (1 \leqslant s
    \leqslant k)$}\\
0,&\mbox{if $i \in I$.}\end{array}\right.
    \]
It is clear that $\rho$ is an isomorphism. \end{proof}

The following result is well-known:

\begin{proposition}{\cite[\S6.4]{OrTe1992}} \label{prop:hyperplanearrangement} Suppose $r$, $p$ and
  $n$ are positive integers such that $p \mid r$. Let $G$ be the unitary
  reflection group $G (r, p, n)$ and $\mathcal{A}(G (r, p, n))$ the
  corresponding arrangement of reflecting hyperplanes. Then
\[
\mathcal{A}(G (r, p, n)) =
  \left\{\begin{array}{ll} \mathcal{A}_n (r),&\mbox{if $r \neq p$,}\\
\mathcal{A}_n^0 (r),&\mbox{if $r = p$.}\end{array}\right.
\] 
\end{proposition}

\begin{corollary}
  \label{corollary:arrangement}Let $G$ be the unitary reflection group $G (r, p,
  n)$ acting on a complex vector space $V$ of dimension $n$. Then
\[
\mathcal{S}_1^V (G (r, p, n)) \cong \left\{\begin{array}{ll}Q_n (r),&\mbox{if $r \neq p$, $r \neq 1$}\\
Q_n (r,
  \{1\}),&\mbox{if $r = p \neq 1$}\\
\Pi_n,&\mbox{if $r=p=1$}.\end{array}\right.
\]
Furthermore, this isomorphism is $G$-equivariant with respect to the
  actions described in \cite[\S2.9]{Koonin2012-5} and
  {\S}\ref{subsection:Dowling}.
  \end{corollary}
  \begin{proof}
    For the isomorphism, the case $r = 1$ is well-known. The case $r = 2$ is
    discussed in \cite{Lehrer1988} and \cite{BjWa2004/5}, and the general case follows the same argument.
    The $G$-equivariance is clear from the explicit definition of each of
    the actions.
  \end{proof}

Recall the definition of the M\"obius function of a poset:

\begin{definition}\label{definition:Mobius}Let $P$ be a locally finite poset.  The {\em M\"obius function} of $P$ is a
function\\ $\mu_P : P \times P \rightarrow \mathbb{Z}$ defined recursively as
follows:

\begin{alignat*}{2}
\mu_P (x, x) &= 1 \quad&\mbox{for all $x\in P$},\\
\mu_P (x, y) &= - \sum_{x \leqslant z \leqslant y}
\mu_P (x, z) \quad &\mbox{for all $x < y \in P$},\\
\mu_P (x, y) &= 0 \quad &\mbox{otherwise}.
\end{alignat*}
\end{definition}

The following lemma is the reason for our interest in the M\"obius function:

\newcommand{\rank}{\mathrm{rank}}
\begin{lemma}
  \label{lemma:MobiusCM}If a poset $P$ is Cohen-Macaulay over $\mathbb{Z}$ and the rank of $P$ is n, then
\[
\mu ( \hat{P}) = \rank ( \widetilde{H}_n (P, \mathbb{Z})).
\]
\end{lemma}

\begin{proof}
  This follows immediately from the Euler-Poincar\'e formula and the definition of a Cohen-Macaulay poset.
\end{proof}

  \begin{proposition}{\cite[\S6.4]{OrTe1992}}
    \label{proposition:arrangement}
$\;$
    \begin{itemize}
    \item[(i)] The M\"obius function of $\mathcal{A}_n (r) \cong Q_n (r)$ is $(- 1)^n
    (r + 1) (2 r + 1) \ldots ((n - 1) r + 1)$.
\item[(ii)]  The M\"obius function of $\mathcal{A}_n^0 (r) \cong Q_n (r, \{1\})$ is
    $(- 1)^n (r + 1) (2 r + 1) \ldots ((n - 2) r + 1) ((n - 1) (r - 1)) .$
\item[(iii)] The M\"obius function of $\Pi_n$ is $(- 1)^{n - 1} (n - 1) !$.  
\end{itemize}
  \end{proposition}

\subsection{Posets of Eigenspaces for the Groups $G (r, p,
n)$}\label{subsection:G(r,p,n)posets}

Recall from {\S}\ref{subsection:posetsofeigenspaces} that if $\gamma G$ is a reflection
coset in $V =\mathbb{C}^n$, and $\zeta \in \mathbb{C}^{\times}$ is a complex
root of unity then the poset $\mathcal{S}_{\zeta}^V (\gamma G)$ is defined to
be the set $\{V (g, \zeta) \mid g \in \gamma G\}$, partially ordered by
reverse inclusion. There is a close connection between the poset
$\mathcal{S}_{\zeta}^V (\gamma G)$, where $G$ is the reflection group $G (r,
p, n)$, and subposets of the Dowling lattices defined in
{\S}\ref{subsection:Dowling}.

Suppose $G = G (r, p, n)$ acts on $V =\mathbb{C}^n$, and that $\gamma G$ is a
reflection coset in $V$. By \cite[Proposition 3.13]{BrMaMi1999}, up to scalar
multiples either $\gamma = \tmop{diag} (\xi_{\frac{er}{p}}, 1, \ldots, 1)$, where $e
\in \mathbb{N}$, $e \mid p$ and $\xi_m$ denotes a fixed primitive $m$-th root
of unity, or $\gamma$ falls into one of three exceptional cases (see also \cite[Table D.5, p.278]{LeTa2009}). The first of these possibilities is considered in the following theorem.  Note that it is sufficient
to state the results for the case where this scalar multiple is 1, since if
$\alpha$ is a scalar matrix then $\mathcal{S}_{\zeta}^V (\alpha \gamma G) =\mathcal{S}_{\alpha^{- 1} \zeta}^V (\gamma G)$. 

\begin{theorem}
\label{theorem:dowling}Let $G$ be the group $G (r, p, n)$ acting
  on $V =\mathbb{C}^n$, $\gamma G$ a reflection coset in $V$ with $\gamma =
  \tmop{diag} (\xi_{\frac{er}{p}}, 1, \ldots, 1)$ as above, and $\zeta$ a
  primitive $m$-th root of unity. Let $d = \frac{m}{gcd(m, r)}$. Then
\begin{itemize}
\item[(i)]  If $d > 1,$ $n \equiv 0   (\tmop{mod} d)
\tmop{and} (\zeta^n \xi_{\frac{er}{p}}^{- 1})^{\frac{r}{p}} \neq 1,$ then
$\mathcal{S}_{\zeta}^V (\gamma G) \cong Q_n (dr, d, d, \{0\})
\backslash \{ \hat{0} \}.$
\item[(ii)] If $d > 1$ and the conditions of (i)
are not satisfied, then $\mathcal{S}_{\zeta}^V (\gamma G) \cong
Q_n (dr, d, \text{} d) \backslash \{ \hat{0} \}$.
\item[(iii)] If $d = 1$ and $(\zeta^n
\xi_{\frac{er}{p}}^{- 1})^{\frac{r}{p}} \neq 1$, then $\mathcal{S}_{\zeta}^V
(\gamma G) \cong Q_n (r, \{0\}) \backslash \{ \hat{0} \}.$ 
\item[(iv)] If $d = 1, r = p \neq 1, \text{and \
$\zeta^n = \xi_e$} $, then $\mathcal{S}_{\zeta}^V (\gamma G) \cong Q_n (r, \{1\})$.
\item[(v)] If $d = 1,$ $r = p = 1$, then $m = 1$ and
$\mathcal{S}_{\zeta}^V (\gamma G) \cong Q_{n - 1} (1) \cong
\Pi_n$.
\item[(vi)] If $d = 1$ and the conditions of (iii)-(v) are
not satisfied, then $\mathcal{S}_{\zeta}^V (\gamma G) \cong Q_n
(r)$.
  \end{itemize}
\end{theorem}  

  \begin{proof}
    The proof of this theorem occupies the next several pages. \ Before
  embarking on the proof, we provide an outline. \ The first step is to
  describe the eigenspace $V (x, \zeta)$ for $x \in \gamma G$. \ This is Lemma
  \ref{lemma:eigenspaces}. \ Next, we define a map $\tau :
  \mathcal{S}_{\zeta}^V (\gamma G) \rightarrow Q_n (dr)$ which associates to
  each eigenspace $V (x, \zeta)$ an element of the Dowling lattice $Q_n (dr)
  .$ \ It will be clear by construction that $\tau (
  \mathcal{S}_{\zeta}^V (\gamma G)) \subseteq Q_n (dr, d) .$ \ In fact (Lemma
  \ref{lemma:Q_n(drdd)}) it is true that $\tau ( \mathcal{S}_{\zeta}^V (\gamma
  G)) \subseteq Q_n (dr, d, d) .$ \ Up to this point the proof is case free. \
  The next series of Lemmas (until Lemma \ref{lemma:partition}) provide
  restrictions on the size of the image of $\tau$. \ These are case dependent.
  \ Finally it is shown that the image of $\tau$ is in fact as large as
  possible, subject to the restrictions given in the previous lemmas.
    
We first describe explicitly the map $\Gamma : \gamma G \rightarrow
    \text{$\mathcal{S}_{\zeta}^V (\gamma G)$}$ which sends $x \in \gamma G$ to
    $V (x, \zeta)$. Recalling that $\gamma$ is a diagonal matrix, it makes
    sense to write $x = ( \Omega, \sigma)$, where $\Omega = (\omega^{i_1},
    \omega^{i_2}, \ldots, \omega^{i_n}) \in (\mu_r)^n$, $\omega$ is a
    primitive $r$-th root of unity, and $\sigma \in \tmop{Sym} (n)$. Note
    that since $G = G (r, p, n)$ and $\gamma = \tmop{diag}
    (\xi_{\frac{er}{p}}, 1, \ldots, 1)$, it is equivalent to say that
    $(\omega^{\sum_{s = 1}^n i_s} \xi^{- 1}_{\frac{er}{p}})^{\frac{r}{p}} =
    1.$ Let $z = (z_1, z_2, \ldots, z_n) \in \mathbb{C}^n,$ and suppose
    that $z \in V (x, \zeta)$.
    
According to the action described in {\S}\ref{subsection:wreath}, if $z
    \in V (x, \zeta)$ then for each $j \in \{1, \ldots, n\}$, $\zeta z_j =
    \omega^{i_j} z_{\sigma^{- 1} (j)}$. Equivalently, $z_j = \zeta^{- 1}
    \omega^{i_j} z_{\sigma^{- 1} (j)}$.
    
Now suppose that $(j_1, j_2, \ldots, j_k)$ is a cycle of $\sigma$ (so that
    $\sigma (j_1) = j_2$ etc...). Then 
\begin{align}      
z_{j_2} &= \zeta^{- 1} \omega^{i_{j_2}} z_{j_1}\nonumber\\
z_{j_3} &= \zeta^{- 1} \omega^{i_{j_3}} z_{j_2}\nonumber\\
      &\vdots \label{equation:dcycle}\\
z_{j_k} &= \zeta^{- 1} \omega^{i_{j_k}} z_{j_{k - 1}}\nonumber\\
z_{j_1} &= \zeta^{- 1} \omega^{i_{j_1}} z_{j_k}\nonumber
    \end{align}
    Combining these equations, $z_{j_1} = \zeta^{- k} \omega^{\sum_{s = 1}^k
    i_{j_s}} z_{j_1} .$ 
    
    These equations have nontrivial solution if and only if $\zeta^k =
    \omega^{\sum_{s = 1}^k i_{j_s}}$.
    
So if nontrivial solutions exist, $\zeta^k$ must be an $r$-th root of unity
    (since $\omega$ is defined as an $r$-th root of unity). That is,
    $(\zeta^k)^r = 1$, and $k r$ is a multiple of $m.$ Thus $k$ must be a
    multiple of $d = \frac{m}{gcd(m, r)} .$

    If the system of equations (\ref{equation:dcycle}) has a nontrivial
    solution, call it a {\em $d$-cycle}. We have have therefore proved
    the following lemma:

    \begin{lemma}
      \label{lemma:eigenspaces}Suppose $G, \gamma$ and $x$ are as above.
      The equations defining $V (x, \zeta)$ consist of some set of
      $d$-cycles involving disjoint sets of variables, together with equations
      $z_{i_s} = 0$ whenever $z_{i_s}$ does not appear in a $d$-cycle. 
    \end{lemma}
    This lemma determines the map $\Gamma$ completely.
    
   The next step is to associate an element $\overline{(I, \pi, \gamma)} \in
    Q_n (dr)$ with each $V (x, \zeta)$ ($x \in \gamma G$). Note that for
    each $s \in \{1, \ldots, k\}$, the complex number $\zeta^{- 1}
    \omega^{i_{j_s}}$ is a ($dr$)-th root of unity. There is a natural way
    to associate an element $\overline{(I, \pi, \gamma)} \in Q_n (dr)$
    with such a $V (x, \zeta) .$ This map is the inverse of the one denoted
    $\rho$ in Proposition \ref{proposition:Dowlingisomorphism}, stated there
    for the special case $m = 1, \gamma = \tmop{Id}$. Write this map as as
    $\tau : \mathcal{S}_{\zeta}^V (\gamma G) \rightarrow Q_n (dr)$. We
    describe $\tau$ explicitly here.

    To define $\tau (V (x, \zeta))$ for $x \in \gamma G$, the first task is to
    describe $\pi$. If $(j_1, j_2, \ldots, j_k)$ is a cycle of $\sigma$,
    $k$ is a multiple of $d$, and $\zeta^k = \omega^{\sum_{s = 1}^k i_{j_s}}$,
    then $\{j_{1,} j_2, \ldots, j_k \}$ forms a nonzero block of $\pi$. \
    Otherwise, $j_1, j_2, \ldots, j_k$ are all in the zero block of $\pi .$

 Now let $B_i =\{j_1, j_2, \ldots, j_k \}$ be a nonzero block of $\pi$. Define $\gamma_i : B_i \rightarrow \mu_{dr}$ by:
\begin{align}
      \gamma (j_1) &= 1\nonumber\\
      \gamma (j_2) &= \zeta^{- 1} \omega^{i_{j_2}} &  &  \nonumber\\
      \gamma (j_3) &= \zeta^{- 2} \omega^{i_{j_2} + i_{j_3}}      \label{equation:bijection} \\
    &\vdots  \nonumber\\
      \gamma (j_k) &= \zeta^{- (k - 1)} \omega^{( \sum_{s = 2}^k i_{j_s})} \nonumber
    \end{align}
    Do this for each nonzero block $B_i$. Starting the cycle with a
    different element has the effect of shifting the labels by a constant
    factor, so the equivalence class $\overline{(I, \pi, \gamma)}$ is
    independent of the ordering of the cycles. This defines the map $\tau :
    \text{$\mathcal{S}_{\zeta}^V (\gamma G)$} \rightarrow Q_n (dr)$.

    It is clear from the definition that $\tau$ is injective. Thus it
    remains to determine the size of the image.

    To start, it is true by construction that $\overline{(I, \pi, \gamma)} \in
    Q_n (dr, d)$.
        
    Recall (Definition \ref{definition:Q_n(r,d,k)}) that $Q_n (r, d, k)$ is
    the `$d$-divisible, $k$-evenly coloured' Dowling lattice.
    
    \begin{lemma}\label{lemma:Q_n(drdd)}
      $\tau ( \mathcal{S}_{\zeta}^V (\gamma G))
      \subseteq Q_n (dr, d, d) .$
    \end{lemma}
    
    \begin{proof}[P{\em roof of Lemma \ref{lemma:Q_n(drdd)}}]
      Setting $r = gcd(m, r) q$, note that every element of $\mathbb{Z}_{dr}$ can
      be written uniquely in the form $aq + bd$ ($0 \leqslant a < d, 0
      \leqslant b < r$). Every element of $\mu_{dr}$ can therefore be
      uniquely written in the form $\zeta^a \omega^b$ ($0 \leqslant a < d, 0
      \leqslant b < r$) via the bijection $\zeta^a \omega^b \leftrightarrow aq
      + bd$. Identifying $\mu_{dr}$ with $\mathbb{Z}_{d r}$ in this way,
      the elements $\zeta^a \omega^b$ ($a$ fixed, $0 \leqslant a < d, 0
      \leqslant b < r$) all lie in the same congruence class modulo $d$. In
      other words, the congruence class modulo $d$ of $\zeta^a
      \omega^b$ depends only on $a$. By construction the labels in each
      block $\overline{(B_i, \gamma_i)}$ are therefore $d$-evenly coloured,
      and so $\overline{(I, \pi, \gamma)} \in Q_n (dr, d, d)$.
    \end{proof}

  Thus far everything has been case free. \ We now consider restrictions to
  the size of the image of $\tau$ in the different cases.
    
    \begin{lemma}
      \label{lemma:dneq1}If $d \neq 1,$ then $\tau ( \mathcal{S}_{\zeta}^V
      (\gamma G)) \subseteq Q_n (dr, d, d) \setminus \{ \hat{0} \}.$
    \end{lemma}
    
    \begin{proof}[P{\em roof of Lemma \ref{lemma:dneq1}}]
      This is clear, since if $d \neq 1$ the unique minimal element $\{ \hat{0} \}$
      is added artificially, in order to make $Q_n (dr, d, d)$ a lattice.
    \end{proof}

 Note that this lemma applies to both cases (i) and (ii). \ In particular if
  the conditions of (ii) are satisfied then $\tau ( \mathcal{S}_{\zeta}^V
  (\gamma G)) \subseteq Q_n (dr, d, d) \backslash \{ \hat{0} \}$ as claimed.
    
    Recall (Definition \ref{definition:Q_n(r,d,k,J)}) that $Q_n (r, d, k, J)$
    is the subposet of $Q_n (r, d, k)$ consisting of elements $\overline{(I,
    \pi, \gamma)}$ whose zero block $I$ satisfies $\left| I \right| \nin J
    \subseteq \{0, 1, \ldots, n\}$, together with the unique minimal element $\{ \hat{0} \}$ of $Q_n
    (r, d, k)$.
    
    \begin{lemma}\label{lemma:ncong0}
      If $n \equiv 0\; (\tmop{mod} \;d)$ and $(\zeta^n
      \xi_{\frac{er}{p}}^{- 1})^{\frac{r}{p}} \neq 1$, then $\tau (
      \mathcal{S}_{\zeta}^V (\gamma G)) \subseteq Q_n (dr, d, d, \{0\})$.
    \end{lemma}
    
    \begin{proof}[P{\em roof of Lemma \ref{lemma:ncong0}}]
      Suppose that the conditions of the lemma hold, and that $\overline{(I,
      \pi, \gamma)} = \tau (V (x, \zeta))$, for some $x \in \gamma G.$ As
      above, write $x = ( \Omega, \sigma)$, where $\Omega = (\omega^{i_1},
      \omega^{i_2}, \ldots, \omega^{i_n}) \in (\mu_r)^n$, $\omega$ is a
      primitive $r$-th root of unity, and $\sigma \in \tmop{Sym} (n)$. As
      remarked previously, since $G = G (r, p, n)$ and $\text{$\gamma =
      \tmop{diag} (\xi_{\frac{er}{p}}, 1, \ldots, 1)$,}$ it follows that
      $(\omega^{\sum_{s = 1}^n i_s} \xi^{- 1}_{\frac{er}{p}})^{\frac{r}{p}} =
      1.$ If $j = (j_1, j_2, \ldots j_k)$ is a cycle of $\sigma$ which
      contributes a nonzero block to $\overline{(I, \pi, \gamma)}$, it follows
      by the argument preceeding Lemma \ref{lemma:eigenspaces} that the
      cycle-product $\omega^{\sum_{s = 1}^k i_{j_s}} = \zeta^k .$ \
      Considering all nonzero blocks together, $\omega^{\sum_{s \nin I}^{}
      i_s} = \zeta^{n - \left| I \right|} .$ Hence if $\left| I \right| = 0,$ it
      follows that $\omega^{\sum_{s = 1}^n i_s}_{} = \zeta^n$. Hence
      $(\zeta^n \xi_{\frac{er}{p}}^{- 1})^{\frac{r}{p}} = 1$, which
      contradicts the hypothesis of the lemma. Thus if $(\zeta^n
      \xi_{\frac{er}{p}}^{- 1})^{\frac{r}{p}} \neq 1$, $\tau (
      \mathcal{S}_{\zeta}^V (\gamma G)) \subseteq Q_n (dr, d, d, \{0\})$ as
      claimed. 
    \end{proof}
    
    Lemma \ref{lemma:ncong0} applies to both cases (i) and (iii). \ Combining
  Lemmas \ref{lemma:dneq1} and \ref{lemma:ncong0}, it follows that if the
  conditions of (i) are satisfied then $\tau ( \mathcal{S}_{\zeta}^V (\gamma
  G)) \subseteq Q_n (dr, d, d, \{0\}) \backslash \{ \hat{0} \}$. \ If the
  conditions of (iii) are satisfied then $Q_n (dr, d, d, \{0\}) = Q_n (r, 1,
  1, \{0\}) = Q_n (r, \{0\})$.  Further, the unique minimal element $\{ \hat{0} \}$ of $Q_n (r, \{0\})$ does not correspond to an eigenspace, since it has zero block of size zero.  Hence $\tau ( \mathcal{S}_{\zeta}^V (\gamma
  G)) \subseteq Q_n (r, \{0\}) \backslash \{ \hat{0} \}$ as claimed.

    \begin{lemma}
      I\label{lemma:r=p}f $n \equiv 1 \;(\tmop{mod} \;d)$, $r = p$, and $\zeta^n =
      \xi_e$, then $m \mid r$ and $\tau ( \mathcal{S}_{\zeta}^V (\gamma G))
      \subseteq Q_n (dr, d, d, \{1\}) = Q_n (r, 1, 1, \{1\}) = Q_n (r, \{1\})
      .$
    \end{lemma}
    
    \begin{proof}[P{\em roof of Lemma \ref{lemma:r=p}}]
      Adopt the notation of the previous lemma. To show that $m \mid r$
      under the assumptions of the lemma, first note that $\zeta^{nr} =
      (\zeta^n)^r = (\xi_e)^r = 1$, since $e \mid p$ by definition, and in
      this case $r = p$. Also since $n \equiv 1 \;(\tmop{mod} \;d)$, we can
      write $n - 1 = jd$ for some $j \in \mathbb{Z}$. Hence $\zeta^{(n - 1)
      r} = \zeta^{(jd) r} = \zeta^{\frac{jmr}{gcd(m, r)}} =
      (\zeta^m)^{\frac{jm}{gcd(m, r)}} = 1$, noting that $\frac{jm}{gcd(m, r)} \in
      \mathbb{Z}$. Hence $\zeta^r = \zeta^{nr} \zeta^{- (n - 1) r} = 1$. \
      Since $\zeta$ is by definition a primitive $m$-th root of unity,
      it follows that $m \mid r$. In particular, $d = 1$.

      For the second part of the lemma, note that as in the proof of the
      previous lemma, $\omega^{\sum_{s \nin I}^{} i_s} = \zeta^{n - \left| I
      \right|}$. Suppose, by way of contradiction, that $\left| I \right| = 1$, and
      in particular that $I =\{t\}$ for some $t \in [n]$. Then
      $\omega^{\sum_{s \neq t} i_s} = \zeta^{n - 1}$. Since $r = p$ we
      know that $\omega^{\sum_{s = 1}^n i_s }= \xi_e$, so therefore
      $\omega^{i_t} = \zeta^{- (n - 1)} \xi_e .$ Since by assumption
      $\zeta^n = \xi_e$, the conclusion is that $\zeta = \omega^{i_t}$. 
            
      Furthermore, since $t$ is the only element of $I$, it follows that
      $(t)$ is a singleton cycle of $\sigma$. By \eqref{equation:dcycle},
      and the fact that $t \in I$, the equation $z_t = \zeta^{- 1} \omega^{i_t}
      z$ has no nontrivial solution. This is a contradiction, since
      $\zeta = \omega^{i_s}$. Hence the assumption that $\left| I \right| = 1$ is
      invalid, and $\tau ( \mathcal{S}_{\zeta}^V (\gamma G)) \subseteq Q_n
      (dr, d, d, \{1\})$, as claimed. Since $m = r$, it follows that $d =
      1$, and thus $Q_n (dr, d, d, \{1\}) = Q_n (r, 1, 1, \{1\}) = Q_n (r,
      \{1\}),$ completing the proof of Lemma \ref{lemma:r=p}.
    \end{proof}
    
      Lemma \ref{lemma:r=p} applies to case (iv). \ In particular if the
  conditions of (iv) are satisfied then $\mathcal{S}_{\zeta}^V (\gamma G)
  \subseteq Q_n (r, \{1\})$. \ Note that Lemma \ref{lemma:r=p} shows that the
  case $J =\{1\}$ only arises when $d = 1$. \ Hence there is no corresponding
  case for $d > 1$.
    
    \begin{lemma}
       \label{lemma:partition}If $d = 1, r = p = 1, \tmop{then} m = 1 \tmop{and}
    \text{$\tau ( \mathcal{S}_{\zeta}^V (\gamma G)) \subseteq Q_n (1, [n])
    \cong \Pi_n \cong Q_{n - 1} (1)$} .$
    \end{lemma}
      
\begin{proof}[P{\em roof of Lemma \ref{lemma:partition}}]
       First, it is clear that if $d = 1$
      then $m \mid r$, and so $m = 1$. \ Since $m = 1, \zeta = 1$ also. \ If
      the conditions of the lemma are satisfied, then so are the conditions of
      Lemma \ref{lemma:r=p}. \ Hence $\tau ( \mathcal{S}_{\zeta}^V (\gamma G))
      \subseteq Q_n (r, \{1\}) = Q_{_n} (1, \{1\}) .$ \ However since $r = 1,
      \omega = 1$. \ \ Hence the system of equations (\ref{equation:dcycle})
      always has a nonzero solution. \ Thus $I$ cannot contain any elements at
      all, from which it follows that $\mathcal{S}_{\zeta}^V (\gamma G))
      \subseteq Q_n (1, [n])$. \ It is clear that $Q_n (1, [n]) \cong \Pi_n
      \cong Q_{n - 1} (1)$.
      \end{proof}
          
Hence under the conditions of case (v), $\tau ( \mathcal{S}_{\zeta}^V
  (\gamma G)) \subseteq \Pi_n$. \ Finally, under the conditions of case (vi),
  $Q_n (dr, d, d) = Q_n (r, 1, 1) = Q_n (r)$, and hence by Lemma
  \ref{lemma:Q_n(drdd)}, $\tau ( \mathcal{S}_{\zeta}^V
  (\gamma G)) \subseteq Q_n (r) .$
 
   It remains to show that, in each of cases (i)-(vi) of the theorem, that
  there are no further restrictions on the image of $\tau$. \ First note that
  if $m = 1, \gamma = \tmop{Id}$ the statement reduces to Corollary
  \ref{corollary:arrangement}. \ In this case, $Q_n (r, 1, 1) = Q_n (r)$ and
  $Q_n (r, 1, 1, \{1\}) = Q (r, \{1\})$. \ Also if $r = p = m = 1$ then
  $\gamma = \tmop{Id}$ is the only possibility since $\frac{er}{p} = 1$, so
  (v) is true. \

    \ Given an element $\overline{(I, \pi, \gamma)} \in Q_n (dr, d, d)$, \ we
  must find some $x = (\Omega, \sigma) \in \gamma G (r, p, n)$ such that $\tau
  (V (x, \zeta)) = \overline{(I, \pi, \gamma)}$ \ Recall that $\pi = B_1 \mid
  \cdots \mid B_k$ is a partition of $[n] \backslash I$. \ Suppose $B_l
  =\{j_1, \ldots, j_{k_l} \}$ is a nonzero block of $\pi .$ \ \ We may assume
  that the map $\gamma_i$ has the form described in equations
  (\ref{equation:bijection}), reordering the $j_i$ if necessary:
  \begin{eqnarray}
    \gamma (j_1) = 1 &  &  \nonumber\\
    \gamma (j_2) = \zeta^{- 1} \omega^{i_{j_2}} &  &  \nonumber\\
    \gamma (j_3) = \zeta^{- 2} \omega^{i_{j_2} + i_{j_3}} &  &  \nonumber\\
    \ldots &  &  \nonumber\\
    \gamma (j_{k_l}) = \zeta^{- (k_l - 1)} \omega^{( \sum_{s = 2}^{k_l}
    i_{j_s})} &  &  \nonumber
  \end{eqnarray}
  Define $(j_1 \ldots, j_{k_l})$ to be a cycle of $\sigma$, and $\Omega_{j_s}
  \assign \omega^{i_{j_s}} (2 \leqslant s \leqslant k_l)$ while $\Omega_{j_1}
  \assign$ $\zeta^{k_l} \omega^{- \sum_{s_{} = 2}^{k_l} i_{j_s}} .$ \ Do this
  for each nonzero block $B_l (1 \leqslant l \leqslant k)$. \ Suppose also
  that $I =\{i_1, \ldots, i_q \}.$ \ If $\left| I \right| = 0$ then this data
  completely specifies $x.$ \ Note however that $x \in \gamma G (r, p, n)
  \Leftrightarrow (\zeta^n \xi_{\frac{er}{p}}^{- 1})^{\frac{r}{p}} = 1,$ which
  is precisely the condition in parts (i) and (iv) of the theorem. \

    If $\left| I \right| \neq 0$, first suppose $m \neq 1$. Define $(i_t)   (1
    \leqslant t \leqslant q)$ to be singleton cycles of $\sigma$. We
    must choose $\Omega_{i_t} \neq \zeta$. If $r \neq p$ there is no problem
    doing this - define $\Omega_{i_t} \assign 1 \neq \zeta (1 \leqslant t
    \leqslant l - 1)$. There are then $\frac rp$ possible choices for
    $\Omega_{i_l}$, and at least one of these is not equal to $\zeta$.

    If $r = p$, then since $m \neq 1$, define $\Omega_{i_t} \assign 1 \neq
    \zeta$ ($1 \leqslant t \leqslant l - 1$). If $\left| I \right| = 1$ then as
    above we are forced to define $\Omega_{i_1} \assign \zeta^{- (n - 1)}
    \xi_e$. This is not equal to $\zeta$ unless $\zeta^n =
    \xi_{\frac{er}{p}} = \xi_e$. This is the condition for part (iii) of the
    theorem.

    If $r = p$ and $\left| I \right| > 1$, there are two cases. If $\zeta^n \neq
    \xi_e$ then define $\Omega_{i_t} \assign 1 \neq \zeta$ ($1 \leqslant t
    \leqslant q - 1$), $\Omega_{i_q} \assign \zeta^{- (n - l)} \xi_e$. If
    $\zeta^n = \xi_e$ let $(i_t) (1 \leqslant t \leqslant q - 2)$ be singleton
    cycles of $\sigma$, and $(i_{q - 1}, i_q)$ a cycle of length 2. Define
    $\Omega_{i_t} \assign 1 \neq \zeta$ ($1 \leqslant t \leqslant q - 1$),
    $\Omega_{i_q} \assign \zeta^{- (n - l)} \xi_e$ as before. Since $\zeta^2
    \neq \zeta^{- (n - l)} \xi_e$, the equations which contribute to the zero
    block of $\overline{(I, \pi, \gamma)}$ do indeed have only the trivial
    solution. This completes the proof when $m \neq 1.$

    Now suppose $m = 1$. If $\xi_{\frac{er}{p}} \neq 1$, let $(i_1, i_2,
    \ldots, i_q)$ be a cycle of $\sigma$ of length $q$ and set $\Omega_{i_t}
    \assign 1$ ($1 \leqslant t \leqslant q - 1$), $\Omega_{i_q} \assign
    \xi_{\frac{er}{p}}$. Since $1 = \zeta^l \neq \xi_{\frac{er}{p}}$, the
    relevant equations have only the trivial solution, as required. If
    $\xi_{\frac{er}{p}} = 1$ then $\gamma G = G$, and the problem reduces to
    Corollary $\ref{corollary:arrangement}$.

    In each of the above cases it is easy to verify that $\tau (V (g, \zeta))
    = \overline{(I, \pi, \gamma)}$, as required.  This completes the proof of Theorem \ref{theorem:dowling}.
  \end{proof}

  \begin{corollary}
    \label{corollary:CM}The poset $\mathcal{S}_{\zeta}^V (\gamma G (r, p,
    n))$ is Cohen-Macaulay over $\mathbb{Z}$,
   and over any field $\mathbb{F}$.
      \end{corollary}

\begin{proof}
      Each of the subposets of Dowling lattices which appear in Theorem \ref{theorem:dowling} admit a recursive atom ordering, by the results in \S\ref{subsection:Dowling}.  The corollary now follows immediately from Proposition
      \ref{proposition:CMimplicationswachs} and Theorem \ref{theorem:recursive} in the case where $\gamma$ acts diagonally.

    The three exceptional cases (see \cite[Table D.5, p.278]{LeTa2009}) can be treated by direct computation.  A detailed proof can be found in \cite[Remark 4.2.9]{Koonin2012}.

\end{proof}
  
  \begin{corollary}
If $x, y \in \gamma G (r, p, n)$ and $\gamma = \tmop{diag} (\xi_{\frac{er}{p}}, 1, \ldots, 1)$ then for some $z \in \gamma G
    (r, p, n)$, $V (x, \zeta) \cap V (y, \zeta) = V (z, \zeta)$.
    \end{corollary}

\begin{proof}
      Let $M = V (x, \zeta) \cap V (y, \zeta)$. The defining equations for
      $M$ are those for $V (x, \zeta)$ together with those for $V (y, \zeta)
      .$ Using the isomorphism established in the proof of Theorem \ref{theorem:dowling}, combining these two
      systems of equations corresponds precisely to taking joins in Dowling
      lattices. Hence $M$ is itself an eigenspace, as required.
    \end{proof}
  
  By a direct computation, we could verify this corollary for all reflection cosets associated to imprimitive groups (including the exceptional cosets).  In fact, this corollary holds general reflection cosets, not just the imprimitive ones.  This is proven in \cite[Corollary 3.2]{Koonin2012}.

\subsection{Dimensions of Representations for the Imprimitive
Groups}\label{subsection:dimensions}

The posets $\mathcal{S}_{\zeta}^V (\gamma G (r, p, n))$ are all Cohen-Macaulay
over $\mathbb{Z}$, and over any field $\mathbb{F}$ (Corollary
\ref{corollary:CM}). However all of these posets have a unique maximal
element, and possibly a unique minimal element as well, so their homology is
uninteresting -- see the remarks following Definition \ref{definition:Stilde_v(zeta)}. Hence our interest lies in the poset $\widetilde{\mathcal{S}}_{\zeta}^V
(\gamma G (r, p, n))$, which (recall Definition \ref{definition:Stilde_v(zeta)}) is
obtained from $\mathcal{S}_{\zeta}^V (\gamma G (r, p, n))$ by removing the
unique maximal element, and the unique minimal element if it exists.

\begin{remark}
  It is possible \label{remark:fullspace}to say precisely when the full space $V$ is the unique minimal element of
  $\mathcal{S}_{\zeta}^V (\gamma G (r, p, n))$.  Assume once more that $\gamma = \tmop{diag} (\xi_{\frac{er}{p}}, 1, \ldots, 1).$
\begin{align*}
\mbox{The full space $V \in \mathcal{S}_{\zeta}^V (\gamma G (r, p, n))$}&\Leftrightarrow \tmop{diag} (\zeta, \zeta, \ldots \zeta) \in \gamma G (r, p,
  n)\\
&\Leftrightarrow \zeta^n = \xi_{\frac{er}{p}} \xi_{\frac{r}{p}}\mbox{
  (this notation defined in {\S}\ref{subsection:G(r,p,n)posets})}\\
  &\qquad\qquad \mbox{ and $\zeta$ is an $r$-th root of unity}\\  
&\Leftrightarrow (\zeta^n \xi_{\frac{er}{p}}^{-
  1})^{\frac{r}{p}} = 1 \mbox{ and } d = 1.
  \end{align*}
  
\end{remark}

Where homology is taken over a field $\mathbb{F}$, we would like
information about the homology representation of $G$ on the unique
non-vanishing reduced homology of $\widetilde{\mathcal{S}}_{\zeta}^V
(\gamma G (r, p, n))$. At a minimum, we would like to know the dimension of
such a representation. If $P$ is any poset, define $\hat{P}$ to be the poset obtained from $P$ by adjoining a
unique minimal element $\hat{0}$ and a unique maximal element $\hat{1}$ (see \cite[\S2.2]{Koonin2012-5}'. 

\begin{proposition}
  Let $G$ be the group $G (r, p, n)$ acting on $V =\mathbb{C}^n$, $\gamma G$
  a reflection coset in $V$, and $\zeta$ a primitive $m$-th root of unity.
  Then $\dim \widetilde{H}_{\tmop{top}} ( \widetilde{\mathcal{S}_{}}_{\zeta}^V
  (\gamma G (r, p, n)) = \mu (
  \widehat{\widetilde{\mathcal{S}}_{\zeta}^V (\gamma G (r, p,
  n))}$. 
\end{proposition}

\begin{proof}
  This is a special case of Lemma \ref{lemma:MobiusCM}.
\end{proof}

Theorem \ref{theorem:dowling} can be reformulated for the poset
$\widehat{\widetilde{\mathcal{S}}_{\zeta}^V (\gamma G (r, p,
n))}$.  The case $n=1$ is trivial as if $V$ is one-dimensional, the poset $\mathcal{S}_{\zeta}^V (\gamma G
    (r, p, n))$ is either the one-element poset or the two-element totally ordered poset according to whether or not $V \in \mathcal{S}_{\zeta}^V (\gamma G  (r, p, n))$ (see Remark \ref{remark:fullspace}).  In the following proposition, assume that $n > 1$:

\begin{proposition}\label{proposition:imprimitiveMobius}
  Let $G$ be the group $G (r, p, n)$ acting on $V =\mathbb{C}^n$, $\gamma G$
  a reflection coset in $V$ with $\gamma =  \tmop{diag} (\xi_{\frac{er}{p}}, 1, \ldots, 1)$, and $\zeta$ a primitive $m$-th root of unity.
  Let $d = \frac{m}{gcd(m, r)}$, and suppose $n > 1$.  Then
\begin{itemize}
\item[(i)] If $n \equiv 0 \;(\tmop{mod} \;d)$ and $(\zeta^n \xi_{\frac{er}{p}}^{- 1})^{\frac{r}{p}} \neq 1,$ then
  $\widehat{\widetilde{\mathcal{S}}_{\zeta}^V (\gamma G)} \cong
  Q_n (dr, d, d, \{0\})$.
\item[(ii)] If $d = 1, r = p \neq 1, \text{and \
$\zeta^n = \xi_e$} $, then $\widehat{\widetilde{\mathcal{S}}_{\zeta}^V (\gamma G)} \cong Q_n (r, \{1\})$.
\item[(iii)]  If $d = 1,$ $r = p = 1$, then $m = 1$ and
$\widehat{\widetilde{\mathcal{S}}_{\zeta}^V (\gamma G)} \cong Q_{n - 1} (1) \cong
\Pi_n$.
\item[(iv)] In all other cases,
  $\widehat{\widetilde{\mathcal{S}}_{\zeta}^V (\gamma G)} \cong
  Q_n (dr, d, d)$.
\end{itemize}  
\end{proposition}  
  
  \begin{proof}
The poset $\widehat{\widetilde{\mathcal{S}}_{\zeta}^V
    (\gamma G (r, p, n))}$ is isomorphic to $\mathcal{S}_{\zeta}^V (\gamma G
    (r, p, n))$ if the latter has a unique minimal element, and is equal to $\mathcal{S}_{\zeta}^V (\gamma G
    (r, p, n))$ with a unique minimal element adjoined otherwise.  In the statement of Theorem \ref{theorem:dowling}, $\widehat{\widetilde{\mathcal{S}}_{\zeta}^V
    (\gamma G (r, p, n))}$ does not have a unique minimal element in cases (i), (ii) and (iii) when $n > 1$, while it does in the remaining cases.  The result now follows immediately.
  \end{proof}

Hence in order to compute the dimension of the
unique non-vanishing reduced homology of $\widetilde{\mathcal{S}}_{\zeta}^V
(\gamma G (r, p, n))$ in all cases of Proposition \ref{proposition:imprimitiveMobius}, it will suffice to compute the M\"obius
function of the posets $Q_n (dr, d, d)$, $Q_n (dr, d, d, \{0\})$, and $Q_n (r,
\{1\})$. The last of these posets is isomorphic to $\mathcal{L}(G (r, r,
n))$ (Corollary \ref{corollary:arrangement}), so its M\"obius function is
known (Proposition \ref{proposition:arrangement}). The following section
deals with the first two cases in question.

\section{Exponential Dowling Structures}

An explicit formula for the dimension of the representation of $G$ on the
unique non-vanishing reduced homology of $\widetilde{\mathcal{S}}_{\zeta}^V
(\gamma G (r, p, n))$ (or equivalently, the M\"obius function of
$\widehat{\widetilde{\mathcal{S}}_{\zeta}^V (\gamma G (r, p,
n))})$ is not known in all cases. However generating functions for these
M\"obius functions can be found. The key results in this section are
Corollary \ref{corollary:dowlingmobius} and Corollary
\ref{corollary:dowlingmobius0}. The posets $\mathcal{S}_{\zeta}^V (\gamma G
(r, p, n))$ are examples of {\em exponential Dowling structures}. This
class of posets was introduced by Ehrenborg and Readdy in \cite{EhRe2009} as a
generalisation of the {\em exponential structures} first described by
Stanley \cite{Stanley1979}. Exponential structures are a general family of
posets modelled on the partition lattice $\Pi_n$. They allow the study of
generating functions associated to a sequence of posets $\mathbf{R}=\{R_1,
R_2, \ldots\}$ of the form $\sum_{n = 0}^{\infty} \frac{f (n) x^n}{n!M (n)}$ in
a mechanical way. The integers $M (n)$ have a combinatorial meaning - they
are the number of minimal elements in the poset $R_n$.

\begin{definition}{\cite[Definition 3.1]{EhRe2009}}  
An {\em exponential structure} $\mathbf{R}=(R_1, R_2,
  \ldots)$ is a sequence of posets such that:
\begin{itemize}  
\item[(E1)] The poset $R_n$ has a unique maximal element $\hat{1}$ and every
  maximal chain in $R_n$ contains $n$ elements.
\item[(E2)] For an element $x \in R_n$ of rank $k$, the interval $[x, \hat{1}]$
  is isomorphic to the partition lattice on $n - k$ elements, $\Pi_{n - k}$.
\item[(E3)] The lower order ideal generated by $x \in R_n$ is isomorphic to
  $R_1^{a_1} \times R_2^{a_2} \times \ldots \times R_n^{a_n}$. Define the
 {\em type} of $x$ to be the $n$-tuple $(a_1, a_2, \ldots,
  a_n)$. 
\item[(E4)] The poset $R_n$ has $M (n)$ minimal elements. The sequence $(M (1),
  M (2), \ldots)$ is called the {\em denominator sequence} of
  $\mathbf{R}$.
  \end{itemize}
\end{definition}

The generalisation to Dowling lattices due to Ehrenborg and Readdy is as
follows.

\begin{definition}{\cite[Definition 3.2]{EhRe2009}}\label{definition:expdowlingstructure}
 Let $G'$ be a finite
  group. An {\em exponential Dowling structure} $\mathbf{S}=(S_0, S_1, \ldots)$
  associated to an exponential structure $\mathbf{R}=(R_1, R_2, \ldots)$
  is a sequence of posets such that:
\begin{itemize}
\item[(D1)] The poset $S_n$ has unique maximal element $\hat{1}$ and every
  maximal chain in $S_n$ contains $n + 1$ elements.
\item[(D2)] For an element $x \in S_n, [x, \hat{1}]$ is isomorphic to the Dowling
  lattice $Q_{n - \tmop{rk} (x)} (G')$.
\item[(D3)] Each element in $S_n$ has a type $(b ; a_1, \ldots a_n)$ assigned
  such that the lower order ideal generated by $x$ is isomorphic to $S_b
  \times R_1^{a_1} \times \ldots \times R_n^{a_n}$.
\item[(D4)] The poset $S_n$ has $N (n)$ minimal elements. The sequence $(N (0),
  N (1), \ldots)$ is called the {\em denominator sequence} of $\mathbf{S}$. 
  \end{itemize}
\end{definition}

Note that $S_0$ is the poset with one element and that $N (0) = 1$.

\begin{remark}
  In this thesis, we consider the special case where $G'$ is a cyclic group.
\end{remark}

The utility of knowing that a sequence is an exponential (Dowling) structure
lies in the fact that M\"obius functions are easy to compute.  Recall (Definition \ref{definition:Mobius}) that for a bounded poset $P$, we define $\mu (P) \assign \mu ( \hat{0}, \hat{1}).$

\begin{theorem}{\cite[Corollary 3.8]{EhRe2009}}\label{theorem:exponentialdowling}
Let $\mathbf{S}=(S_0, S_1, \ldots)$ be an exponential Dowling
  structure with denominator sequence $(N (0), N (1), \ldots)$ and associated
  exponential structure $\mathbf{R}=(R_1, R_2, \ldots)$ with denominator
  sequence $(M (1), M (2), \ldots) .$ Then the M\"obius function of the
  posets $R_n \cup \{ \hat{0} \}$, respectively $S_n \cup \{ \hat{0} \}$, has
  the generating function:
\begin{align*}
    \sum_{n \geq 1} \mu (R_n \cup \{ \hat{0} \}) \cdot \frac{x^n}{M (n) \cdot
    n!} &= - \ln \left( \sum_{n \geq 0} \frac{x^n}{M (n) \cdot n!} \right) \\
    \sum_{n \geq 0} \mu (S_n \cup \{ \hat{0} \}) \cdot \frac{x^n}{N (n) \cdot
    n!} &= - \left( \sum_{n \geq 0} \frac{x^n}{N (n) \cdot n!} \right) \left(
    \sum_{n \geq 0} \frac{(s \cdot x)^n}{M (n) \cdot n!} \right)^{-
    \frac{1}{s}}, 
\end{align*}
where $s=\left|G\right|$. 
\end{theorem}

Let $\mathbf{R}$ be an exponential structure and $d$ a positive
integer. Stanley defined the exponential structure $\mathbf{R}^{(d)}$ by
letting $R_n^{(d)}$ be the subposet of $R_{dn}$ of all elements $x$
of type $(a_1, a_2, \ldots)$ where $a_i = 0$ unless $d$ divides
$i$.

In \cite[Example 3.4]{EhRe2009}, Ehrenborg and Readdy define an analogous
exponential Dowling structure. Let $\mathbf{S}$ be an exponential Dowling
structure associated with the exponential structure $\mathbf{R}$. Let
$S_n^{(d, e)}$ be the subposet of $S_{dn + e}$ consisting of all elements
$x$ of type $(b ; a_1, a_2, \ldots)$ such that $b \geqslant e$, $b
\equiv e \;(\tmop{mod}\; d)$ and $a_i = 0$ unless $d$ divides $i$. Then $\mathbf{S}^{(d, e)} = (S_0^{(d, e)}, S_1^{(d, e)}, \ldots)$ is an
exponential Dowling structure associated with the exponential structure
$\mathbf{R}^{(d)}$. The minimal elements of $S_n^{(d, e)}$ are the
elements of $S_{dn + e}$ having types given by $b = e$, $a_d = n$ and $a_i
= 0$ for $i \neq d.$

If $S_n$ is the Dowling lattice $Q_n (r)$ and $e < d$, then $S_n^{(d, e)}$ is
the poset $Q_{dn + e} (r, d) \setminus \{ \hat{0} \}$. We show that the
posets $Q_{dn + e} (r, d, k) \setminus \{ \hat{0} \}$ form an exponential
Dowling structure, though the underlying exponential structure is different. In the following proposition, the group $G'$ of Definition
\ref{definition:expdowlingstructure} is the cyclic group $\mu_r$.

\begin{definition}\label{definition:R'_n}
  Suppose $k, d \in \mathbb{Z}$ with $k \mid d$. Let
  $c_n = \overline{(I_n, \pi_n, \gamma_n)}$ be the element of $Q_{dn} (r, d,
  k)_{}$ defined as follows:
\begin{align*}
I_n &=\{0\},\\
\pi_n &=\{1, 2, \ldots, dn\},\\
\gamma (i) &= \lfloor \frac{(i - 1) k}{dn} \rfloor + 1\mbox{ for } i = 1, 2,
  \ldots, dn.
\end{align*}
Define $\left.R'\right._n^{(d, k)}$ to be the lower order ideal generated by $c_n$ in
  $Q_{dn} (r, d, k)$.\\ Define $\mathbf{R}'^{(d, k)} = (\left.R'\right._1^{(d, k)},\left.R'\right._2^{(d, k)},\ldots).$\\ 
Define $\mathbf{S}^{(d, e, k)} =(S_0^{(d, e, k)}, S_1^{(d, e, k)}, \ldots)$, where $S_n^{(d, e, k)}$ is
    the poset $Q_{dn + e} (r, d, k) \setminus \{ \hat{0} \}$.
\end{definition}  

\begin{lemma}\label{lemma:expstructure}
 Let $\mathbf{R}'^{(d, k)} = (\left.R'\right._1^{(d, k)},\left.R'\right._2^{(d, k)},\ldots).$
 be defined as in Definition
    \ref{definition:R'_n}. Then $\mathbf{R}'^{(d, k)}$ is an exponential
    structure. 
  \end{lemma}
  
  \begin{proof}
    It is routine to check axioms (E1)--(E4). The minimal elements of
    $\left.R'\right._n^{(d, k)}$ are the elements lying below $c_n \in Q_{dn} (r, d, k)$
    having type $(b ; a_1, a_2, \ldots)$, where $b = 0, a_d = n$, and $a_i = 0$ for $i \neq d.$
  \end{proof}
  
  \begin{lemma}\label{lemma:expdowlingstructure}
    Let $\mathbf{S}^{(d, e, k)}$ be defined as above.  Then
    $\mathbf{S}^{(d, e, k)}$ is an exponential Dowling structure associated
    with the exponential structure $\mathbf{R}'^{(d, k)}$.
  \end{lemma}

\begin{proof}
  It is routine to check axioms (D1)--(D4) for $\mathbf{S}^{(d, e,
  k)}$. The minimal elements of $S_n^{(d, e, k)}$ are the elements of $Q_{dn
  + e} (r, d, k) \setminus \{ \hat{0} \}$ having types given by $b = e, a_d =
  n$, and $a_i = 0$ for $i \neq d.$
\end{proof}

\begin{lemma}\label{lemma:atoms}
With $\mathbf{S}^{(d, e, k)},
  \mathbf{R}'^{(d, k)}$ defined as above,
\begin{align*}
    M (n) &= \frac{[( \frac{dn}{k}) !]^k}{n! [( \frac{d}{k}) !]^{kn}}\\
    N (n) &= \frac{(dn + e) !r^{(d - 1) n}}{n!e! [( \frac{d}{k}) !]^{kn}
    k^{dn}} 
  \end{align*}
 \end{lemma}
  
  \begin{proof}
    To compute $M (n)$, recall, as in the proof of Proposition
    \ref{lemma:expstructure}, that the minimal elements of $\left.R'\right._j^{(d)}$ are
    the elements lying below $c_n \in Q_{dn} (r, d, k)$ having type $(b ; a_1,
    a_2, \ldots)$, where $b = 0, a_d = n, \tmop{and} a_i = 0 \tmop{for} i \neq
    d.$ \

    Consider the number of ways of choosing such an atom. First note that
    labelling of elements in the atom is completely determined by the
    labelling of $x$. Therefore the only freedom in choosing an atom comes
    from the way the elements of the single nonzero block of size $dn$ in $x$
    are split up. Within this nonzero block, there are $\frac{dn}{k}$
    elements of each of $k$ colours. Pick any colour, and split these
    $\frac{dn}{k}$ elements into $n$ sets of $\frac{k}{d}$ colours. This can
    be done in $\frac{1}{n!} \tbinom{\frac{dn}{k}}{\frac{d}{k}, \frac{d}{k},
    \ldots, \frac{d}{k}}$ ($n$ $\frac{d}{k}$'s in the denominator)
    ways, since the order in which these sets are chosen is unimportant. \
    Since there are $k$ colours and we have to split the $\frac{dn}{k}$
    elements of each colour up in this way, this gives $\Bigl( \frac{1}{n!}
    \binom{\frac{dn}{k}}{\frac{d}{k}, \frac{d}{k}, \ldots, \frac{d}{k}}\Bigr)^k$
    ways of splitting the colours. We then have to form nonzero blocks of
    size $d$ in the atom by combining one set of size $\frac{d}{k}$ of each
    colour. This can be done in $(n!)^{k - 1}$ ways. Hence the total
    number of such atoms is given by
\begin{align*}
M (n) &=\left( \frac{1}{n!}
    \binom{\frac{dn}{k}}{\frac{d}{k}, \frac{d}{k}, \ldots, \frac{d}{k}}\right)^k
    (n!)^{k - 1}\\
&= \frac{[( \frac{dn}{k})
    !]^k}{n! [( \frac{d}{k}) !]^{kn}},
    \end{align*}as claimed.
       
    To compute $N (n)$, first pick the zero block (of size $e$) and
    the $n$ nonzero blocks (of size $d$) in $\frac{1}{n!}
    \binom{dn + e}{e, d, \ldots d}$ ($n$ $d$'s in the denominator)
    ways.

    Now count the number of ways of labelling the elements within each nonzero
    block. It is most convenient to ignore equivalence of labellings at
    first, and then divide by $r$ at the end to get the
    required answer. First partition the $d$ elements into
   $k$ sub-blocks of size $\frac{d}{k}$. This can be
    done in $\binom{d}{\frac{d}{k}, \frac{d}{k}, \ldots \frac{d}{k}}$
    ($k$ $\frac{d}{k}$'s in the denominator) ways. 
    Elements in the same sub-block have labels in the same congruence class
    modulo $k$. There are now $\frac{r}{k}$ choices for
    each label. After dividing by $r$, the number of ways of
    labelling each nonzero block of size $d$ is $\frac{d!r^{d -
    1}}{[( \frac{d}{k}) !]^k k^d}$.

    There are $n$ such blocks, so
\begin{align*}
      N (n) &= \frac{1}{n!} \binom{dn + e}{e, d, \ldots d} \left[ \frac{d!r^{d
      - 1}}{[( \frac{d}{k}) !]^k k^d} \right]^n\\
      &= \frac{(dn + e) !r^{(d - 1) n}}{n!e! [( \frac{d}{k}) !]^{kn} k^{dn}}, 
    \end{align*}
    as claimed.  \end{proof}
  
  For the purposes of this work, the main interest lies in the posets $Q_{dn +
  e} (dr, d, d)$ (and subposets with restricted zero block size).

  In the following lemma, the group $G'$ of Definition
  \ref{definition:expdowlingstructure} is the cyclic group $\mu_{dr}$.

  \begin{lemma}\label{lemma:atoms2}
    Let $S_n^{(d, e, d)} \tmop{be} \tmop{the} \tmop{poset} Q_{dn + e} (dr, d,
    d) \setminus \{ \hat{0} \}$. Then $M (n) = (n!)^{d - 1}$\\ and $N (n) = \frac{(dn + e) !r^{(d - 1)
    n}}{n!e!d^n}$.
  \end{lemma}
  
  \begin{proof}
    This is a special case of the previous lemma, obtained by replacing $r$
    with $dr$, and $k$ with $d$.
  \end{proof}

\begin{corollary}
    \label{corollary:dowlingmobius}
\[
\sum_{n \geq 0} \mu (Q_{dn + e} (dr, d, d))
      \cdot \frac{x^n e!d^n}{(dn + e) !r^{(d - 1) n}} = - \left( \sum_{n \geq
      0} \frac{x^n e!d^n}{(dn + e) !r^{(d - 1) n}} \right) \left( \sum_{n \geq
      0} \frac{x^n d^n r^n}{(n!)^{d - 1}} \right)^{- \frac{1}{dr}}
\]
\end{corollary}
    
    \begin{proof}
      Apply Theorem \ref{theorem:exponentialdowling} to the exponential
      Dowling structure $\mathbf{S}^{(d, e, d)} = (S_0^{(d, e, d)}, S_1^{(d,
      e, d)}, \ldots)$, where $S_n^{(d, e, d)} = Q_n (dr, d, d) \setminus \{
      \hat{0} \}$. The formula for $M(n)$ and $N(n)$ is given in Lemma
      \ref{lemma:atoms2}. 
    \end{proof}
    
    \begin{example}
      Consider the case $r = p = 1$, $\gamma = \tmop{Id}$, $\zeta = - 1$ in
      Theorem \ref{theorem:dowling}. Case (ii) is the applicable one. Thus
      $\mathcal{S}_{- 1}^V (A_{n - 1}) \cong Q_n (2, 2, 2) \backslash \{ \hat{0} \}$. Suppose
      further that $n=2k$ is even.

      Corollary \ref{corollary:dowlingmobius} then gives:
\begin{align*}
\sum_{k \geqslant 0} \mu (S_{- 1}^V (A_{2 k - 1})) \cdot \frac{2^k
      x^k}{(2 k) !} &= - \Bigl( \sum_{k \geqslant 0} \frac{2^k x^k}{(2 k) !}\Bigr) \Bigl(
      \sum_{k \geqslant 0} \frac{2^k x^k}{(k!)^2}\Bigr)^{- \frac{1}{2}}\\
&= -
      \Bigl(1 + x + \frac{1}{6} x^2 + \frac{1}{90} x^3 + \frac{1}{2520} x^4 +
      \frac{1}{113400} x^5 + O (x^6)\Bigr)\\
&\qquad\times\Bigl(1 - x + x^2 - \frac{10}{9} x^3 + \frac{95}{72} x^4 -
      \frac{2927}{1800} x^5 + O (x^6)\Bigr)\\
&= -
      \Bigl(1 + \frac{1}{6} x^2 - \frac{4}{15} x^3 + \frac{51}{140} x^4 -
      \frac{136}{2835} x^5 + O (x^6)\Bigr).
      \end{align*}
            This gives $\mu (S_{- 1}^V (A_1)) = 0$, $\mu (S_{- 1}^V (A_3)) = - 1$, $\mu
      (S_{- 1}^V (A_5)) = 24$, $\mu (S_{- 1}^V (A_7)) = - 918$ and $\mu
      (S_{- 1}^V (A_9)) = 54560$.  A separate but similar calculation gives the values of the M\"obius functions for $n$ odd.

    \end{example}

Ehrenborg and Readdy also developed a theory of M\"obius
    functions of {\em restricted structures}. The following material is taken directly from \cite{EhRe2009}.

    \begin{definition}
Let $I$ be a subset of $\mathbb{N}^+$, and
      $\mathbf{R}= (R_1, R_2, \ldots)$ an exponential structure. Define the {\em restricted poset $R_n^I$} to
      be the poset consisting of all elements $x \in R_n$ whose type $(a_1,
      \ldots, a_n)$ satisfies the condition that $a_i > 0$ implies $i
      \in I$. For $n \in I$ let $\mu_I (n)$ denote the M\"obius function of
      the poset $R_n^I \cup \{ \hat{0}$\}. For $n \nin I$ let $\mu_I (n) = 0$.
    \end{definition}
    
Suppose $n \in \mathbb{N}$ and define
\[
m_n : = \sum_{x \in R_n^I \cup \{ \hat{0} \}} \mu_I (
    \hat{0}, x) =1 + \sum_{x \in R_n^I} \mu_I (^{} \hat{0}, x).
    \]
    
\begin{theorem}{\cite[Theorem 4.1]{EhRe2009}}
\[
\sum_{i \in I} \mu_I (i) \frac{x^i}{M (i) \cdot i!} = - \ln
      \left( \sum_{n \geqslant 0} \frac{x^n}{M (n) \cdot n!} - \sum_{n \nin I}
      m_n \cdot \frac{x^n}{M (n) \cdot n!} \right).
 \]
\end{theorem}
      
There is a corresponding theorem for exponential Dowling
      structures:

      \begin{definition}
Let $I\subseteq \mathbb{N}$, and
        $J\subseteq \mathbb{Z}_{\geqslant 0}$. For an
        exponential Dowling structure $\mathbf{S}= (S_0, S_1, \ldots)$
        define the {\em restricted poset} $S_n^{I, J}$ to be all
        elements $x \in S_n$ whose type $(b ; a_1,
        \ldots a_n)$ satisfies $b \in J$ and $a_i > 0$ implies $i
        \in I$. For $n \in J$ define $\mu_{I,
        J} (n)$ to be the M\"obius function of the poset $S_n^{I, J} \cup \{
        \hat{0} \}$. For $n\nin J$ define $\mu_{I, J} (n) = 0.$
      \end{definition}
      
      For $n \in \mathbb{Z}_{\geqslant 0}$ define
\[
p_n \assign \sum_{x \in S_n^{I, J} \cup \{ \hat{0} \}}
      \mu_{I, J} ( \hat{0}, x) = 1 + \sum_{x \in S_n^{I, J}} \mu_{I, J} (
      \hat{0}, x).
      \]
 Note that if $n \in J$ then the poset $S_n^{I, J}$ has a maximal
      element, so $p_n = 0$.

      \begin{theorem}{\cite[Theorem 4.2]{EhRe2009}}
\label{theorem:restricted}
\[
\sum_{b \in J} \mu_{I, J} (b) \cdot \frac{x^b}{N (b) \cdot
        b!} = - \sum_{n \geqslant 0} \frac{x^n}{N (n) \cdot n!} +
        \sum_{n \nin J} p_n \cdot \frac{x^n}{N (n) \cdot n!}\left( \sum_{n
        \geqslant 0} \frac{(r \cdot x)^n}{M (n) \cdot n!} - \sum_{n \nin I}
        m_n \cdot \frac{(r \cdot x)^n}{M (n) \cdot n!}
        \right)^{\frac{1}{r}}
        \]
      \end{theorem}

      \begin{corollary}{\cite[Corollary 4.3]{EhRe2009}}
        \label{corollary:restricted2}
        Let $I \subseteq \mathbb{N}$ be a semigroup and $J \subseteq
        \mathbb{Z}_{\geqslant 0}$ such that $I + J \subseteq J$. Then the
        M\"obius function of the restricted poset $R_n^I \cup \{ \hat{0}\}$
        and $S_n^{I, J} \cup \{ \hat{0}\}$ respectively has the generating
        function
\begin{align*}
\sum_{n \in I} \mu_I (n) \cdot \frac{x^n}{M (n) \cdot n!} &= - \ln
        \left( \sum_{n \in I \cup \{0\}} \frac{x^n}{M (n) \cdot n!} \right),\\
\sum_{n \in J} \mu_{I, J} (n) \cdot \frac{x^n}{N (n) \cdot n!} &= -
        \left( \sum_{n \in J} \frac{x^n}{N (n) \cdot n!} \right) \cdot \left(
        \sum_{n \in I \cup \{0\}} \frac{(r \cdot x)^n}{M (n) \cdot n!}
        \right)^{-\frac{1}{r}}.
\end{align*}
\end{corollary}

Corollary \ref{corollary:restricted2} can be used to study the posets $Q_{dn} (dr, d, d, \{0\})$.

      \begin{corollary}
\label{corollary:dowlingmobius0}
        \[
        \sum_{n \geq 0} \mu (Q_{dn} (dr, d, d, \{0\})) \cdot \frac{x^n
        d^n}{(dn) !r^{(d - 1) n}} = - \left( \sum_{n \geq 1} \frac{x^n
        d^n}{(dn) !r^{(d - 1) n}} \right) \left( \sum_{n \geq 0} \frac{x^n
        d^n r^n}{(n!)^{d - 1}} \right)^{- \frac{1}{dr}}
        \]
\end{corollary}
        
        \begin{proof}
          In Corollary \ref{corollary:restricted2} set $\mathbf{R}'^{(d)},
          \mathbf{S}^{(d, e, k)}$ as in Definition \ref{definition:R'_n}, and let $I =\mathbb{N}^{}, J
          =\mathbb{N} \setminus \{0\}$.
        \end{proof}

\section{Acknowledgements}
The author would like to acknowledge his supervisor Gus Lehrer and associate supervisor Anthony Henderson for their encouragement, patience, generosity and enthuiasm.

During preparation of this work, it became evident that Alex Miller had obtained many of the same results independently.  The author would like to acknowledge this
and thank Alex for useful discussions on this topic.

\bibliographystyle{plain}

\bibliography{bibliography2}

\end{document}